\documentclass[11 pt]{amsart}

\usepackage[margin=2.5cm]{geometry}
\usepackage{amsfonts,graphicx,enumerate}
\usepackage{amssymb, mathrsfs}
\usepackage{amsmath}
\usepackage{latexsym}
\usepackage{mathrsfs}
\usepackage{amsthm}
\usepackage{verbatim}
\usepackage{amscd}
\usepackage{hyperref}

\usepackage{comment}
\usepackage{framed}
\usepackage{color}
\definecolor{shadecolor}{gray}{0.875}

\newtheorem{thrm}{Theorem}[section]
\newtheorem{lem}[thrm]{Lemma}
\newtheorem{cor}[thrm]{Corollary}
\newtheorem{prop}[thrm]{Proposition}
\newtheorem{conj}[thrm]{Conjecture}

\newtheorem*{thrma}{Theorem A}
\newtheorem*{thrmb}{Theorem B}
\newtheorem*{thrmc}{Theorem C}

\theoremstyle{definition}
\newtheorem{defn}[thrm]{Definition}

\newtheorem{ex}[thrm]{Example}
\newtheorem{rmk}[thrm]{Remark}
\newtheorem*{remark}{Remark}

\DeclareMathOperator{\st}{\ \big|\ }

\DeclareMathOperator{\bl}{Bl}

\DeclareMathOperator{\sh}{\varepsilon\,}
\DeclareMathOperator{\fuj}{\mu\,}

\DeclareMathOperator{\Eff}{\overline{Eff}}
\DeclareMathOperator{\Mov}{\overline{Mov}}
\DeclareMathOperator{\Nef}{{Nef}}

\DeclareMathOperator{\Supp}{Supp\,}

\DeclareMathOperator{\Sing}{Sing\,}

\DeclareMathOperator{\vol}{vol\,}

\DeclareMathOperator{\mult}{mult}
\DeclareMathOperator{\Null}{Null\,}

\let\cal\mathcal
\let\frak\mathfrak
\let\bb\mathbb

\input xy
\xyoption{all}

\newcommand{\factor}[2]{\left. \raise 1pt\hbox{\ensuremath{#1}} \right/
        \hskip -2pt\raise -3pt\hbox{\ensuremath{#2}}}

\numberwithin{equation}{thrm}

\begin{document}

\title{Seshadri constants for curve classes}
\author{Mihai Fulger}
\address{Department of Mathematics, University of Connecticut, Storrs, 341 Mansfield Road, Unit 1009, Storrs, CT 06269}
\address{EPFL/SB/TAN, Station 8, CH-1015 Lausanne, Switzerland}
\address{Institute of Mathematics of the Romanian Academy, P. O. Box 1-764, RO-014700,
Bucharest, Romania}
\email{mihai.fulger@uconn.edu}

\begin{abstract}
We develop a local positivity theory for movable curves on projective varieties similar to the classical Seshadri
constants of nef divisors. We give analogues of the Seshadri ampleness criterion,
of a characterization of the augmented base locus of a big and nef divisor, and of the interpretation of Seshadri constants as an asymptotic measure of jet separation.
\end{abstract}

\maketitle

\section{Introduction}

Let $X$ be a projective variety of dimension $n$ over an algebraically closed field.
Fix $L$ a \emph{nef} class in the real N\' eron--Severi space $N^1(X)$, and a closed point $x\in X$. The Seshadri constant $\sh(L;x)$ is a measure of local positivity for $L$ at $x$. It can be defined as

\begin{equation}
\label{seshdefdiv}
\sh(L;x):=\max\bigl\{t\geq 0\ |\ \pi^*L-tE\mbox{ is a nef divisor}\bigr\},
\end{equation}
where $\pi:\bl_xX\to X$ denotes the blow-up of $x$, and $E$ is the natural exceptional divisor. An equivalent interpretation is 

\begin{equation}
\label{seshdefdivmult}
\sh(L;x)=\inf\left\{\frac {L\cdotp C}{\mult_xC}\st C\mbox{ reduced irreducible curve on }X\mbox{ through }x\right\}.
\end{equation}

\noindent The Seshadri constant $\sh(L;x)$ is a homogeneous numerical invariant of $L$. Its properties are carefully detailed in \cite{dem,primer} and \cite[Chapter 5]{laz04}).

\vskip.25cm
In this paper we study a similar local positivity measure for curve classes.
Inspired by \eqref{seshdefdivmult} we consider the following

\begin{defn}\label{defn:seshdef}Let $C$ be a $1$-cycle on $X$, and fix a closed point $x\in X$. Set
$$\sh(C;x):=\inf\left\{\frac{C\cdotp L}{\mult_xL}\ \big|\ L\mbox{ effective Cartier divisor on }X\mbox{ through }x\right\}.$$
\end{defn}
 
\noindent This is a numerical homogeneous invariant of $C$. 
It is non-negative when the numerical class of $C$ is in the \emph{movable} cone of curves $\Mov_1(X)$, dual to the cone of effective divisors in $N^1(X)$. In the sequel we focus on such movable classes.
When $X$ is singular, the regularity condition that $L$ be Cartier is necessary for defining $C\cdotp L$.
We will later look at a similar notion that takes into account all Weil divisors through $x$.

\begin{rmk}The function $\sh(\cdot;x):\Mov_1(X)\to\bb R$
is 1-homogeneous, nonnegative, and concave. It is positive and locally uniformly continuous on the strict interior of the cone.
\end{rmk}

An analogous interpretation to \eqref{seshdefdiv} is

\begin{prop}\label{prop:seshbypullcurve}Let $x\in X$ be a \emph{smooth} point on a projective variety. Let $\ell$ be a line in the exceptional divisor $E\simeq\bb P^{n-1}$ of the blow-up $\pi:\bl_xX\to X$. Let $C\in\Mov_1(X)$.
Then $$\sh(C;x)=\max\bigl\{t \st\ \pi^*C-t\ell\in\Mov_1(\bl_xX)\bigr\}.$$
\end{prop}

\noindent The condition that $x$ be a smooth point of $X$ is necessary for constructing a good pullback $\pi^*C$ as in Definition \ref{def:pullcurve} that respects numerical equivalence. 

\subsection{Examples}

\begin{ex}[Surfaces are a familiar picture] Let $X$ be a smooth projective surface.
There is a canonical identification $N^1(X)\simeq N_1(X)$.
Through it, being nef is the same as being movable.
In this case the theory of Seshadri constants is the same for curves as it is for divisors.
\end{ex}

\begin{ex}[Projective space, Example \ref{ex:pn}] Let $X=\bb P^n$, let $\Lambda$ be the class of a line, and $H$ the class of a linear hyperplane. Then $\sh(\Lambda;x)=\sh(H;x)=1$ for all $x\in X$.
\end{ex}

\begin{ex}[Smooth toric varieties, Example \ref{ex:toric}] Let $X=X(\Delta)$ be a smooth projective toric variety, and let $x_{\sigma}\in X$ be a torus-invariant point corresponding to an $n$-dimensional regular cone $\sigma\in\Delta$.
If $C\in\Mov_1(X)$, then 
$$\sh(C;x_{\sigma})=\min\bigl\{C\cdotp D_{\tau}\st \tau\mbox{ ray of }\sigma\bigr\}.$$ Here $D_{\tau}$ is the torus-invariant divisor corresponding to the ray $\tau$. A similar formula holds classically for $\sh(L;x_{\sigma})$ if $L$ is a nef divisor.
\end{ex}

\begin{ex}[Picard rank 1, Example \ref{ex:picard1}] Let $X$ be a smooth projective variety such that ${\rm rank}\, N^1(X)=1$. Let $H$ be an ample generator of $N^1(X)$. The curve intersection class $H^{n-1}$ generates $N_1(X)$. Seshadri constants (for curves and for divisors) determine the nef and the pseudo-effective cones of divisors for the blow-up $\pi:\bl_xX\to X$ with exceptional divisor $E$.
\begin{center}
\def\svgwidth{10cm}
\begingroup%
  \makeatletter%
  \providecommand\color[2][]{%
    \errmessage{(Inkscape) Color is used for the text in Inkscape, but the package 'color.sty' is not loaded}%
    \renewcommand\color[2][]{}%
  }%
  \providecommand\transparent[1]{%
    \errmessage{(Inkscape) Transparency is used (non-zero) for the text in Inkscape, but the package 'transparent.sty' is not loaded}%
    \renewcommand\transparent[1]{}%
  }%
  \providecommand\rotatebox[2]{#2}%
  \ifx\svgwidth\undefined%
    \setlength{\unitlength}{566.92913386bp}%
    \ifx\svgscale\undefined%
      \relax%
    \else%
      \setlength{\unitlength}{\unitlength * \real{\svgscale}}%
    \fi%
  \else%
    \setlength{\unitlength}{\svgwidth}%
  \fi%
  \global\let\svgwidth\undefined%
  \global\let\svgscale\undefined%
  \makeatother%
  \begin{picture}(1,1)%
    \put(0,0){\includegraphics[width=\unitlength,page=1]{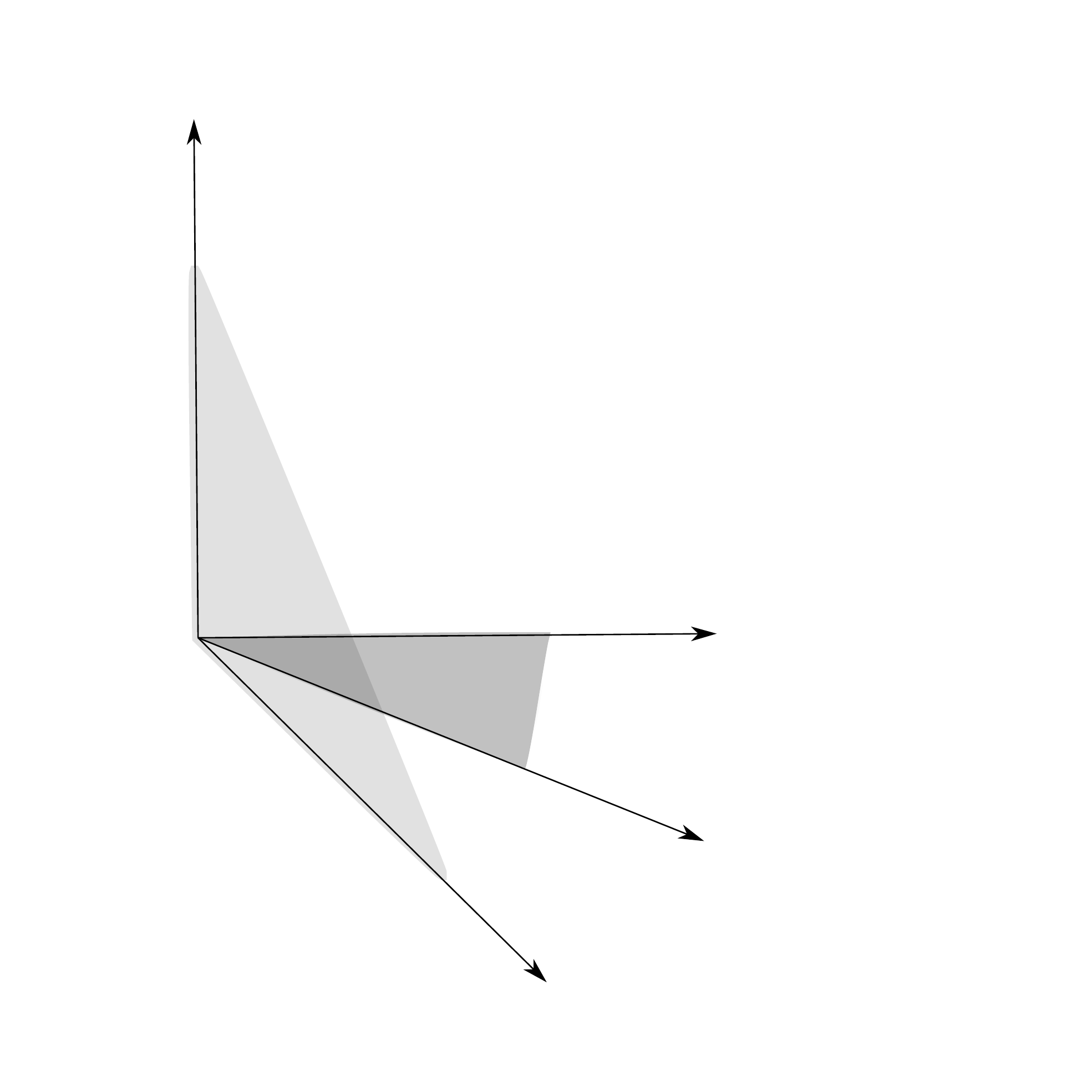}}%
    \put(0.20410715,0.85422619){\color[rgb]{0,0,0}\makebox(0,0)[lb]{\smash{$E$}}}%
    \put(0.54806545,0.51404762){\color[rgb]{0,0,0}\makebox(0,0)[lb]{\smash{}}}%
    \put(0.61232143,0.44979164){\color[rgb]{0,0,0}\makebox(0,0)[lb]{\smash{$\pi^*H$}}}%
    \put(0.61610119,0.2608036){\color[rgb]{0,0,0}\makebox(0,0)[lb]{\smash{$\pi^*H-\sh(H;x)E$ }}}%
    \put(0.51404766,0.11339294){\color[rgb]{0,0,0}\makebox(0,0)[lb]{\smash{$\pi^*H-\frac{(H^n)}{\sh(H^{n-1};x)}E$}}}%
    \put(0,0){\includegraphics[width=\unitlength,page=2]{Conepic.pdf}}%
    \put(0.24568453,0.48002975){\color[rgb]{0,0,0}\makebox(0,0)[lb]{\smash{$\Eff^1(\bl_xX)$}}}%
    \put(0.42711311,0.35151787){\color[rgb]{0,0,0}\makebox(0,0)[lb]{\smash{$\Nef^1(\bl_xX)$}}}%
  \end{picture}%
\endgroup%

\end{center}

\noindent Recall that the \emph{pseudo-effective} cone of divisors $\Eff^1(X)$ is the closure in $N^1(X)$ of the convex cone generated by classes of effective Cartier divisors.\qed
\end{ex}

A particular case of this brings to light a first difference between curves and divisors.

\begin{ex}[Grassmann varieties, Example \ref{ex:grass}] Let $X=G(k;n)$ be the Grassmann variety of $k$-dimensional subspaces of $\bb C^n$. Let $\Lambda\subset X$ be a line generating $N_1(X)$. Let $H$ be a hyperplane section of $X$ in its Pl\" ucker embedding. Then for all $x\in X$ we have $\sh(H;x)=1$ (\cite[Example 5.1.7]{laz04}), but $$\sh(\Lambda;x)=\frac 1{\min\{k,n-k\}}.$$
\end{ex}

\begin{rmk}There are a couple of reasons why this may be surprising.
\begin{itemize}
\item If $L$ is a very ample or just ample and globally generated divisor on a projective variety, then $\sh(L;x)\geq 1$ for all $x\in X$ (\cite[Example 5.1.18]{laz04}). While $\Lambda\in G(k;n)$ satisfies any reasonable analogue of ampleness and most immediate analogues of global generation, the corresponding result does not hold.
\item As $k$ grows, we see that $\sup_{x\in X}\sh(C;x)$ can be arbitrarily small. For big and nef divisors on projective varieties over $\bb C$ it is conjectured that $\sh(L;x)\geq 1$ for very general $x\in X$. 
\end{itemize}
\end{rmk}

\begin{ex}[Genus 3 curve in its Jacobian, Example \ref{ex:g3r1}] Let $C$ be a curve of genus 3.
Let $X=J(C)$ be its Jacobian, and assume that ${\rm rank}\, N^1(X)=1$.
This holds for very general (non-hyperelliptic) curves, and for curves that are very general among the hyperelliptic ones. 
Let $C\subset X$ be an Abel--Jacobi inclusion. For all $x\in X$, we show that
$$\sh(C;x)=\frac 32.$$
For comparison, by a computation of \cite{kong},
$$\sh(\theta;x)=\left\{\begin{array}{cl}\frac{12}7 &\mbox{ in the non-hyperelliptic case}\\ \frac 32 &\mbox{ in the hyperelliptic case}\end{array}\right.,$$
where $\theta$ is a principal polarization.
\end{ex}

\begin{remark}In genus 2, under the same assumption that the Jacobian has Picard rank 1, \cite{stef98} proves that $\sh(C;x)=\sh(\theta;x)=\frac 43$.
\end{remark}

\subsection{The main results}
\begin{thrma}[``Ampleness'' criterion]\label{Seshadricriterion}Let $X$ be a projective variety over an algebraically closed field, 
and let $C\in\Mov_1(X)$.
Then $C$ is in the strict interior of the movable cone if and only if $\inf_{x\in X}\sh(C;x)>0$.
\end{thrma}

The motivation comes from a result of Seshadri (\cite[Chapter 10]{Har70} or \cite[Theorem 1.4.13]{laz04}). It states that a nef divisor $L$ is ample (that is in the strict interior of the nef cone) if and only if $\inf_{x\in X}\sh(L,x)>0$.
By \cite{bdpp13}, $\Mov_1(X)$ is the dual of the pseudo-effective cone of divisors $\Eff^1(X)\subset N^1(X)$. 
Thus we may see the movable cone of curves as the ``nef'' cone of curves. 

The proof of Theorem A relies on an analysis of Zariski decompositions for divisors, and on multiplicity estimates for divisors in large linear series.
By comparison, the proof of the classical Seshadri criterion for divisors is inductive, relying on the Nakai--Moishezon criterion of ampleness, so it does not readily extend to curves. 

\begin{thrmb}[A characterization of the Null locus]
\label{thm:nulllocus}Let $X$ be \underline{\emph{smooth}} projective over an algebraically closed field. 
Let $C\in\Mov_1(X)$ and assume there exists $x_0\in X$ such that $\sh(C;x_0)>0$. 
Then there exist at most finitely many irreducible divisors $L_1,\ldots,L_r$ on $X$ such that $C\cdotp L_i=0$ for $1\leq i\leq r$.
Furthermore, the ``null locus'' ${\rm Null}\, (C):=L_1\cup\ldots\cup L_r$ coincides with the set
$\bigl\{x\in X\st \sh(C;x)=0\bigr\}$. In particular, the latter is Zariski closed in $X$.
\end{thrmb}

\noindent Manifestly $\sh(C;x)=0$ whenever $x$ belongs to the support of some effective divisor $L$ with $C\cdotp L=0$. The content of the theorem is the reverse inclusion $\bigl\{x\in X\st \sh(C;x)=0\bigr\}\subseteq\Null(C)$.
The motivation comes from \cite[Corollary 5.6 and Remark 6.5]{ELMNP}. 
There the authors prove that if $L$ is a big and nef divisor, 
then the locus $\bigl\{x\in X\st \sh(L;x)=0\bigr\}$
coincides with the non-ample locus $B_+(L)$, also known as the augmented base locus, 
and with the union
$$\bigcup_{{\small\begin{array}{c}V\subset X\mbox{ closed subvariety, }\dim V>0\\ L|_V\mbox{ not big }(\Leftrightarrow L^{\dim V}\cdotp V=0)\end{array}}}V.$$
This is a good analogue of Theorem B, since for curves the only nontrivial restrictions to consider are to divisors.

\begin{cor}[Example \ref{ex:POL}]If $X$ is a smooth complex projective variety and $C\subset X$ is a smooth curve with ample normal bundle, then $[C]$ is big. 
\end{cor}

This was a question of Peternell, also answered positively by \cite{Ottem16,Lau17}.

\begin{thrmc}[``Jet separation'']Let $X$ be a projective variety over an algebraically closed field, and let $x$ be a \underline{\emph{smooth}} point on it. 
Let $[C]$ be an $\bb R$-class in the strict interior of $\Mov_1(X)$. Then $\sh(C;x)$ is the supremum of all $s\geq 0$ such 
that for all effective $\bb R$-Weil $\bb R$-divisors $L$ with $x\in\Supp L$
there exists an effective $\bb R$-1-cycle $C'\equiv C$ such that $\Supp C'$ and $\Supp L$ meet properly and $\mult_xC'\geq s$. The same statement holds with $\bb Q$ replacing $\bb R$ throughout.
\end{thrmc}

In a sense, the condition in the theorem is that the union of the supports of effective $\bb R$-1-cycles of class $[C]$ with multiplicity at least $s$ at $x$ is dense in $X$.
We also want that the components of these supports sit in general position relative to any divisor, which is why the formulation in the theorem was preferred. 

It is classical that with notation as in the theorem we have $C'\cdotp L\geq\mult_xC'\cdot\mult_xL\geq s\cdot\mult_xL$, leading to $\sh(C;x)\geq s$. The deeper content is that the inequality becomes an equality when considering the supremum of such $s$. 
The motivation came from the alternate interpretation of Seshadri constants of divisors as an asymptotic measure of jet separation. 
Recall that a Cartier divisor $L$ is said to separate $s$-jets at a smooth point $x\in X$ if 
the natural map
$$H^0\bigl(X,\cal O_X(L)\bigr)\to H^0\bigl(X,\cal O_X(L)\otimes\factor{\cal O_X}{\frak m_{x}^{s+1}}\bigr)$$
is surjective. We denoted by $\frak m_{x}$ the ideal sheaf of $x\in X$. 
If $L$ separates $s$-jets at $x$, then the linear series $|L|$ verifies the following incidence condition:
for every irreducible curve $C$ through $x$
there exists a member $L'\in|L|$ such that $\mult_xL'\geq s$ and $L'$ meets $C$ properly.
If $L$ is nef, this implies that $\sh(L;x)\geq s$. When $L$ is ample, \cite[Theorem 6.4]{dem} and \cite[Theorem 5.1.17]{laz04} prove that
this inequality becomes an equality asymptotically. Denote by $s(L;x)$ the largest $s$ such 
that $L$ separates $s$-jets at $x$. Then $\sh(L;x)=\lim_{k\to\infty}\frac{s(kL;x)}{k}$ if $L$ is ample. 

In Theorem \ref{thrm:asymptoticjetseparation} we also prove an asymptotic version of Theorem C giving control on the coefficients of $C'$. 

\subsection{Bounds on Seshadri constants}

For divisors, finding global lower bounds on Seshadri constants for ample divisors is an important problem.
We refer to \cite{ekl95,dem} and \cite[Chapter 5]{laz04} for the history of this question
and its relation to the famous Fujita conjecture.

In \cite{el93} (see also \cite[Proposition 5.2.3]{laz04}) it is proved that if $X$ is a smooth projective surface, and $L$ an ample divisor on $X$, then $\sh(L;x)\geq 1$
except for possibly countably many points $x\in X$. A generalization for divisors appears in \cite[Theorem 1]{ekl95}. There it is shown that if $X$ is a smooth projective variety of dimension $n$ and $L$ is big and nef, then $\sh(L;x)\geq\frac 1n$ for very general $x\in X$. It is conjectured that $\sh(L;x)\geq 1$ for very general $x$.

Using the techniques of \cite{el93}, we give lower bounds for complete intersection curves.

\begin{prop}Let $X$ be a projective variety of dimension $n$ over $\bb C$. Let $H$ be an ample
Cartier $\bb Z$-divisor on $X$ such that $(H^n)\geq n^{n-2}$, or $\sh(H;x_0)\geq 1$ for some $x_0\in X$. 
Then $\sh(H^{n-1};x)\geq 1$ for very general $x\in X$.
\end{prop}

The surface case of \cite{el93} is a particular case of this.
See Proposition \ref{prop:lowerboundci} for a sharper version.

\begin{remark}The Example \ref{ex:grass} of the Grassmannian proves that there cannot exist a lower bound independent of dimension that holds for all classes with $\bb Z$-coefficients in the strict interior of $\Mov_1(X)$.
\end{remark}

Upper bounds are easier to find. 

\begin{defn}Let $X$ be a projective variety and fix $x\in X$. For any $L\in\Eff^1(X)$, consider the Fujita--Nakayama-type invariant
$$\fuj(L;x):=\sup\bigl\{t\geq 0\st \pi^*L-tE\in\Eff^1(\bl_xX)\bigr\},$$
where $\pi:\bl_xX\to X$ is the blow-up with exceptional divisor $E$.
\end{defn}

\noindent In a sense this measures the maximal multiplicity at $x$ of effective $\bb R$-divisors $\bb R$-linearly equivalent to $L$.
An important bound on $\fuj(L;x)$ is obtained by counting conditions for a function to vanish at a smooth point with prescribed multiplicity.

\begin{prop}\label{prop:lowerboundmultiplicity}If $x$ is a smooth point of $X$ and $L\in\Eff^1(X)$,
then $$\fuj(L;x)\geq\vol^{1/n}(L).$$
\end{prop}

\noindent From the easy observation that 
\begin{equation}\sh(C;x)\fuj(L;x)\leq C\cdot L\end{equation}
for all $C\in\Mov_1(X)$ and $L\in\Eff^1(X)$ we find
that

\begin{prop}\label{prop:seshadriupperbound}If $x$ is a smooth point of $X$, then
$$\sh(C;x)\leq\inf\left\{\frac{C\cdotp L}{\vol^{1/n}(L)}\st L\mbox{ is a big }\bb R\mbox{-Cartier }\bb R\mbox{-divisor}\right\}.$$
In particular $\sh(H^{n-1})\leq(H^n)^{\frac{n-1}n}$ for every ample divisor class $H$.
\end{prop}

The right hand side is $\mathfrak M^{\frac{n-1}n}(C)$ as defined by
\cite{xi15} and further studied in \cite{lxpos}. The function $\mathfrak M$ is a volume-type function on $\Mov_1(X)$ obtained by polar Legendre--Fenchel transform from the volume function on $\Eff^1(X)$. 

\subsection{Relations with the independent work of \cite{mx17}}Seshadri constants for curves have been studied independently at the same time by \cite{mx17}. They see $\sh(\cdot;x):\Mov_1(X)\to\bb R$
as the polar transform of $\fuj(\cdot;x):\Eff^1(X)\to\bb R$ in the sense that
$$\sh(C;x)=\inf\left\{\frac{C\cdotp L}{\fuj(L;x)}\st L\mbox{ is a big }\bb R\mbox{-Cartier }\bb R\mbox{-divisor}\right\}.$$
There is significant overlap between our work and theirs.
They also prove Theorem A, Theorem B, and Propositions \ref{prop:lowerboundmultiplicity} and \ref{prop:seshadriupperbound}.
Their version of Theorem B is sharper in a sense. It also identifies a movable big $\bb R$-divisor class $L_C$ (an ``$(n-1)$ divisorial root of $C$" coming from \cite[Theorem 1.8]{lxpos}) such that $\Null(C)$ agrees with the union of the divisorial components of ${\bf B}_+(L_C)$. 
The class $C$ can be reconstructed from $L_C$.   

They also consider the polar transform of $\sh(\cdot;x):\Nef^1(X)\to\bb R$ giving rise to a dual function $\fuj(\cdot;x):\Eff_1(X)\to\bb R$.
With notation as in Proposition \ref{prop:seshbypullcurve}, assuming that $x$ is a smooth point, the latter also has a geometric interpretation:
$$\fuj(C;x)=\sup\bigl\{t\geq 0\st \pi^*C-t\ell\in\Eff_1(\bl_xX)\bigr\}.$$
Implicit is the statement that $\pi^*$ preserves the pseudo-effectivity of curves when $\pi$ is the blow-up of a smooth point. This is not true for arbitrary blow-ups (even with smooth centers) because 
the pushforward of a nef divisor is usually only movable. 

\subsection{Sehsadri constants for nef dual classes} Pseudo-effective cones
$\Eff_k(X)\subset N_k(X)$ are defined for all cycle dimensions $0\leq k\leq n$. 
Dually, we have nef cones $\Nef^k(X):=\Eff_k(X)^{\vee}$ for all codimensions inside the dual numerical spaces $N^k(X):=N_k(X)^{\vee}$. For example $\Nef^1(X)$ is the usual cone of nef Cartier divisor classes by \cite{kleiman66}.
When $X$ is smooth, $\Nef^{n-1}(X)=\Mov_1(X)$ by \cite{bdpp13}. 
Working with nef classes allows us to remove the regularity conditions imposed on either $L$ or $x$ in our Definition \ref{defn:seshdef} and Proposition \ref{prop:seshbypullcurve}.

\begin{defn}Let $X$ be a projective variety and fix a possibly singular point $x\in X$. For any $\alpha\in\Nef^k(X)$ set
$$\sh(\alpha;x)=\inf\left\{\frac{\alpha\cdotp V}{\mult_xV}\ \big|\ V\mbox{ effective }k\mbox{-cycle through }x\right\}.$$
\end{defn}
\noindent It is enough to consider the irreducible subvarieties $V$ through $x$, with no regularity conditions. As for divisors, 
$$\sh(\alpha;x)=\max\bigl\{t\geq 0\st \pi^*\alpha+t(-E)^k\in\Nef^k(\bl_xX)\bigr\}.$$
Seshadri constants for nef classes share many of the formal properties we saw for nef divisors or movable curves. The function $\sh(\cdot;x):\Nef^k(X)\to\bb R$
is 1-homogeneous, nonnegative, and concave. It is positive and locally uniformly continuous on the strict interior of the cone.

For nef dual curve classes, meaning $k=n-1$, we find an analogue of Theorem A (cf. Proposition \ref{prop:seshadrinefdualcurves}) and a weaker version of Theorem B (see Proposition \ref{prop:singularnulllocus}). Note that Theorem C does not hold for arbitrary classes in the strict interior of $\Nef^k(X)$. By \cite{delv11,OttemNef}, there exist nef classes that are not pseudo-effective. 
Example \ref {seshnefvsmov} suggests that the Seshadri constants of nef dual classes are finer than those of movable curve classes.

\subsection{Organization} In section 2 we review some background on numerical groups and their duals and on the various positivity notions that we need.
In section 3 we develop the basic properties of Seshadri constants for curves.
The next three sections correspond to the proofs of Theorems A, B, and C respectively. Section 7 deals with bounds on Seshadri constants.
Lastly, we develop a theory of Seshadri constants for nef dual classes in arbitrary codimension.

\vskip.5cm
\begin{paragraph}{{\bf Acknowledgments}} 
Special thanks go to V. Lozovanu for his careful reading of a preliminary version of this work and for his suggestions. We are grateful to N. McCleerey and J. Xiao for kindly sharing their preprint.
The author also thanks A. Bertram, S. Boucksom, B. Lehmann, M. Musta\c t\u a, J.C. Ottem, Z. Patakfalvi, M. Popa, C. Raicu, J. Shin, and J. Waldron for useful discussions. 
\end{paragraph}

\section{Notation and background}

Let $X$ be a projective variety of dimension $n$ over an algebraically closed field.
The real space of Cartier divisors on $X$ modulo numerical equivalence is denoted $N^1(X)$.
Its dual space $N_1(X)$ is the real space of 1-dimensional cycles on $X$ modulo numerical equivalence.
Generalizations for arbitrary dimension are provided by \cite[Chapter 12]{fulton84}.
A $k$-cycle $Z$ is \emph{numerically trivial} if $\int_ZP=0$ for all polynomials $P$ of weight $k$
in Chern classes of possibly several vector bundles on $X$. The real space of cycles modulo
numerically trivial cycles is denoted $N_k(X)$. Its elements are called \emph{numerical cycle classes}. 
Its abstract dual $N^k(X):=N_k(X)^{\vee}$ is also
the space of Chern polynomials of weight $k$ modulo the dual equivalence relation: $P$ is
numerically trivial if $\int_ZP=0$ for all $k$-cycles $Z$. 
The elements of $N^k(X)$ are called \emph{dual cycle classes}.
We routinely denote numerical classes of cycles $V$ or Chern polynomials $P$ by $[V]$ and $[P]$ respectively.
When $X$ is also smooth,
the intersection pairing induces an isomorphism $N^k(X)\simeq N_{n-k}(X)$.
In general we have a natural linear ``cyclification" morphism $N^k(X)\overset{\cap[X]}{\longrightarrow} N_{n-k}(X):\ P\mapsto P\cap[X]$.
It is injective when $k=1$ by \cite[Example 19.3.3]{fulton84}, and surjective when $k=n-1$.
In particular, it is an isomorphism when $n=2$, irrespective of the singularities or normality of $X$.

These vector spaces contain important closed convex cones. 
The closure of the cone generated by effective Cartier divisors on $X$ is the \emph{pseudo-effective cone}
$\Eff^1(X)\subset N^1(X)$. It contains the \emph{nef cone} $\Nef^1(X)$,
the closure of the convex span of ample divisors. 
In the dual space $N_1(X)$ we find the duals of these two cones.
The \emph{Mori cone} $\Eff_1(X)$ is the closure of the cone of effective curve classes, dual to $\Nef^1(X)$ in view of \cite{kleiman66}.
The \emph{movable} cone $\Mov_1(X)$ is the closure of the convex span of irreducible curves that deform in families
that dominate $X$. This cone is dual to $\Eff^1(X)$ by \cite{bdpp13}.

\begin{rmk}The proof of the duality $\Mov_1(X)=\Eff^1(X)^{\vee}$ in \cite{bdpp13} is given for complex projective manifolds. 
\cite[Theorem 11.4.19]{laz04} extends it to projective varieties in characteristic 0.
The key ingredients are the existence of Fujita approximations and Hodge inequalities for intersections of nef divisors. 
These are extended to arbitrary characteristic by \cite{Cutkosky}, even over non-algebraically closed field.
A proof of the existence of Fujita approximations in positive characteristic over an algebraically closed field was previously found by \cite{Takagi}. Newton--Okounkov body techniques are used in \cite{LM09} for a proof valid over any algebraically closed field.
See also \cite[Theorem 2.22]{fl13z}.
\end{rmk}

In $N_k(X)$ we can analogously define the \emph{pseudo-effective cone} 
$\Eff_k(X)$ (\cite{flpos}) and the \emph{movable cone} $\Mov_k(X)$ (\cite{fl13z}).
The \emph{nef cone} $\Nef^k(X)\subset N^k(X)$ (\cite{flpos}) is the dual of $\Eff_k(X)\subseteq N_k(X)$.
We often say that a (dual) cycle is nef/movable
when its numerical class has the same property.

\vskip.25cm

Let $x\in X$ be a closed point, or more generally a $0$-dimensional subscheme, and let $\pi:\bl_xX\to X$ be the blow-up of $X$ at $x$ with exceptional divisor $E$, such that $-E$ is the divisor associated to the relative Serre $\cal O(1)$ sheaf.
By \cite[p.79]{fulton84}, for any irreducible and reduced closed subset $V\subseteq X$ with $\dim V>0$, one has
\begin{equation}\label{multform}\mult_xV=-\overline V\cdotp(-E)^{\dim V},
\end{equation} 
where $\overline V$ denotes the strict transform of $V$. Since the strict transform of a union of subvarieties (different from $\{x\}$) is the union of the strict transforms, one can extend $V\mapsto\overline V$ linearly 
to $k$-cycles with arbitrary coefficients. Then $V\mapsto\mult_xV$ is also linear.

\begin{rmk}[Multiplicity of a divisor along a subvariety]\label{rmk:mdZ}
Let $X$ be a variety, $Z\subset X$ a subvariety of codimension $c\geq 2$, not contained in ${\rm Sing}\, (X)$, and 
$D\subset X$ a Weil divisor.
Let $\pi:\bl_ZX\to X$ be the blow-up with exceptional divisor $E$. The general fiber of $\pi|_E$
is isomorphic to $\bb P^{c-1}$. Let $\ell$ be a line in one such fiber.
Denote by $\overline D$ the strict transform of $D$.
Put 
$$\mult_ZD:=\overline D\cdot\ell.$$
The intersection is well defined, and agrees with the definition from \cite[Section 4.3]{fulton84}:
$\pi_*(\overline D\cdot (-E)^{c-1})=(\mult_ZD)\cdot[Z]$. 
(The intersections and pushforwards happen inside the respective Chow groups and can be computed by restricting $\overline D$ to the Cartier divisor $E$, 
then further restricting to the fiber over general $z\in Z$, which is regularly embedded in $E$. 
Finally, the fiber is smooth, and $\ell$ is a complete curve, 
hence Cartier divisors on it have well defined degree. 
The result clearly does not depend on the choice of the general $z\in Z$.
By the projection formula, the coefficient of $[Z]$
in the pushforward is $\overline D\cdot (-E)^{c-1}\cdot\pi_{|E}^*[z]=\overline D\cdot\ell$, where $z$ is a general point on  $Z$ and $\ell$ a line in $\pi^{-1}z\simeq\bb P^{c-1}$.)

If $X$ and $Z$ are smooth, then $\pi^*D=\overline D+(\mult_ZD)\cdot E$. (Intersect with $\ell$.) The regularity conditions on $X$ and $Z$ are present to guarantee that $\pi^*D-\overline D$ is a multiple of $E$. For possibly singular $X$ and $Z$, if $D$ is Cartier, then $\pi^*D-\overline D-(\mult_ZD)\cdot E$ is an exceptional divisor supported over $(X_{\rm sing}\cap Z)\cup Z_{\rm sing}$.

Since $x\mapsto\mult_x D$ is upper-semicontinuous for the Zariski topology ($x$ is not necessarily closed), we also have $\mult_ZD=\mult_zD$ for general $z\in Z$. This is \cite[Definition 5.2.10]{laz04}.\qed
\end{rmk}

\begin{lem}\label{lem:multiplicitypush} Let $\pi:X\to Y$ be a proper generically finite morphism of varieties.
Let $x\in X$ be a closed point such that $\pi$ is finite in a neighborhood of $x$.
Let $y:=\pi(x)$. Then 
$$\mult_y\pi_*X\geq\mult_xX.$$
More generally, $\mult_y\pi_*X\geq\sum_{i}\mult_{x_i}X$, where $x_i$ ranges though the zero-dimensional non-embedded (primary) components of $\pi^{-1}y$.
\end{lem}

\begin{proof}Let $z:=\pi^{-1}y$ denote the scheme theoretic preimage of $y$.
Assume first that $z$ is a finite length subscheme. 
Let $\tilde \pi:\widetilde X\to\widetilde Y$ be the induced morphism between $\bl_yY$
and $\bl_{z}X$.
Let $E$ denote the exceptional divisor of $\widetilde Y$ and let $F$ denote the exceptional divisor
of $\widetilde X$. We have $\tilde \pi^*E=F$. By $F_x$ we denote the (connected) component of $F$ over $x$. 
It is also a Cartier divisor.
Let $x'$ denote the primary component of $z$ centered at $x$.
For $n:=\dim X$,
we have $\mult_xX\leq\mult_{x'}X=-(-F_x)^n\leq -(-F)^n=-(\deg\pi)\cdot (-E)^n=\deg\pi\cdot\mult_yY=\mult_y\pi_*X$.

When $z$ has components of positive dimension, automatically not through $x$, the inequality
$-(-F_x)^n\leq -(-F)^n$ is unclear. In this case, let $X\overset{f}{\to}Z\overset{g}{\to} Y$
be the Stein factorization of $\pi$. Apply the previous case twice. 
The more general statement is analogous.
\end{proof}

\section{Seshadri constants for curves}

For $C$ a curve cycle and $x\in X$, set
$$\sh(C;x):=\inf\left\{\frac{C\cdotp L}{\mult_xL}\ \big|\ L\mbox{ effective Cartier divisor on }X\mbox{ through }x\right\}.$$

\begin{rmk}From the definition we see that $\sh(C;x)$ depends only on the numerical class of $C$.
Furthermore $[C]\mapsto\sh(C;x)$ is 1-homogeneous on $N_1(X)$, and it is nonnegative and concave on $\Mov_1(X)$.
\end{rmk}

\begin{rmk}\label{rmk:irredenough}If $X$ is smooth, and $C$ is movable, then the infimum can be computed over irreducible divisors $L$.
This is because of the inquality $\frac{a+b}{c+d}\geq\min\left\{\frac ac,\frac bd\right\}$ for $a,b\in\bb R$ and $c,d>0$. When $X$ is singular and $Z$ is a component of $L$, the intersection number $C\cdotp Z$ may be undefined.
\end{rmk}

\begin{ex}[Projective space]\label{ex:pn}Let $\Lambda$ be a line in $\bb P^n$. By B\' ezout, $\frac{\deg L}{\mult_xL}\geq 1$ for every effective divisor $L$ through a point $x\in\bb P^n$. Equality is achieved when $L$ is a linear hyperplane
through $x$. Thus $\sh(\Lambda;x)=1$ for every $x\in\bb P^n$.
\end{ex}
\begin{rmk}[Real coefficients] By the proof of \cite[Lemma 0.14]{fujino09}, we may work with $\bb R$-Cartier $\bb R$-divisors $L$ in the definition of $\sh(C;x)$.
The claim is that being effective for an $\bb R$-Cartier $\bb R$-divisor is equivalent to being a nonnegative combination of (possibly nonreduced and reducible) effective Cartier divisors.
\end{rmk}

\begin{lem}If $C\in N_1(X)$, then $\sh(C;x)\geq 0$ for all $x\in X$ if and only if $C\in\Mov_1(X)$.
\end{lem}
\begin{proof}Immediate from \cite{bdpp13}.
\end{proof}

The next remark shows that at least when we want to restrict to finite Seshadri constants, we do not lose much if we just work with movable classes.

\begin{rmk}\label{rmk:nonmov}If $C\in N_1(X)$ and there exists a base point free Cartier divisor $H$ such that $C\cdotp H\geq 0$ (e.g., when $C$ is effective), 
and there exists an effective Cartier divisor $D$
such that $C\cdotp D<0$, then $\sh(C;x)=-\infty$ for all $x\not\in\Supp D$. 
(Indeed up to replacing $H$ by a divisor passing through $x$, we see that 
$$\inf_{m\to\infty}\frac{C\cdotp(mD+H)}{\mult_x(mD+H)}=-\infty.)$$
\qed
\end{rmk}

The following is the easy implication of a Seshadri-type characterization of the strict interior of $\Mov_1(X)$.

\begin{lem}\label{lem:uniformlowerbound}If $[C]$ is in the strict interior of $\Mov_1(X)$, then there exists a constant $\varepsilon>0$
such that $\sh(C;x)\geq\varepsilon$ for all $x\in X$.
\end{lem}

\begin{proof}
Denote $n:=\dim X$.
There exists
an ample $\bb Q$-divisor $H$ on $X$ and $\alpha\in\Mov_1(X)$ such that $[C]=[H^{n-1}]+\alpha$. 
Let $\pi:\bl_xX\to X$ be the blow-up of $x$ with exceptional divisor $E$.
For all effective Cartier divisors $L$ through fixed $x$,  
$$\frac{C\cdot L}{\mult_xL}\geq\frac{H^{n-1}\cdot L}{\mult_xL}=\frac{\pi^*H^{n-1}\cdot\overline L}{(-E)^{n-1}\cdot(-\overline L)},$$
where $\overline L$ denotes the strict transform of $L$ (it is still Cartier, since $E$ is Cartier).
Since $H$ is ample, by the Seshadri criterion of ampleness (\cite[Chapter 10]{Har70} or \cite[Theorem 1.4.13]{laz04}) there exists $m>0$ such that $\sh(H;x)\geq\frac 1m$ for all $x\in X$ (choosing $m$ such that $mH$ is very ample is enough). Then $\pi^*H-\frac 1mE$ is nef for all $x\in X$ and 
$$0\leq\big(\pi^*H-\frac 1mE\big)^{n-1}\cdot\overline L=\pi^*H^{n-1}\cdot\overline L+\frac 1{m^{n-1}}(-E)^{n-1}\cdot\overline L.$$
Note that the intermediary intersections $(\pi^*H)^{i}\cdot E^{n-1-i}\cdot \overline L$ with $0<i<n-1$ vanish because $\pi^*H\cdot E=0$.
This proves $\sh(C;x)\geq\frac 1{m^{n-1}}$ for all $x\in X$. 
\end{proof}

\begin{cor}The function $\sh(\cdot;x)$ is locally uniformly continuous on the strict interior of $\Mov_1(X)$.
\end{cor}
\begin{proof}The lemma proves that the function is positive on the interior of the movable cone.
Then apply \cite[Lemma 2.7]{Leh16}.
\end{proof}

For an interpretation equivalent to \eqref{seshdefdiv}, we need to be able to pullback. This requires some regularity
condition on $\pi$. 
\begin{defn}\label{def:pullcurve}Let $x\in X$ be a \emph{smooth} point.
In this case $E\simeq\mathbb P^{n-1}$. Fix $\ell$ a line in $E$, so that $E\cdot\ell=-1$. 
For any curve $C\subseteq X$, set
$$\pi^*C:=\overline C+(\mult_xC)\cdot\ell.$$
\end{defn}

\begin{prop}The pullback $\pi^*$ defined above is linear, respects numerical equivalence, and satisfies the projection formula $$D\cdot\pi^*C=(\pi_*D)\cdot C$$ for any Cartier divisor $D$ on $\bl_xX$.
\end{prop} 

\begin{proof}Linearity holds because $V\mapsto\overline V$ is linear, and $V\mapsto\mult_xV$ is linear on cycles of the same dimension. If $D$ is Cartier on $\bl_xX$, then $\pi_*D$ is Cartier on $X$ (because $\pi$
is an isomorphism away from the smooth point $x$) and 
\begin{equation}\label{picsmoothblow}D=\pi^*\pi_*D-(D\cdot\ell)E.\end{equation}
If $C$ is a $1$-cycle, then $\pi^*C\cdot E=0$. 
This is straightforward from the relation $\overline C\cdot E=\mult_xC.$
(see \eqref{multform}) and proves that $\pi^*$ respects numerical equivalence.

From the definition, $\pi_*\pi^*C=C$. Then the projection formula is clear for divisors $D=\pi^*L$.
Using \eqref{picsmoothblow}, it remains to treat the case $D=E$, which follows from $\pi^*C\cdot E=0$ above. 
\end{proof}

\begin{lem}If $x\in X$ is a smooth point and $[C]\in\Mov_1(X)$, then $[\pi^*C]\in\Mov_1(\bl_xX)$.
\end{lem}
\begin{proof}Use the projection formula and the duality between $\Mov_1(\bl_xX)$ and $\Eff^1(\bl_xX)$ (cf. \cite{bdpp13}).
\end{proof}

The next result interprets the Seshadri constant of a curve as the distance from $[\pi^*C]$ to the boundary of $\Mov_1(\bl_xX)$ in the $[-\ell]$ direction:

\begin{prop}\label{prop:seshadriviapull}Let $x\in X$ be a smooth point on a projective variety. Let $C\subseteq X$ be a curve with $[C]\in\Mov_1(X)$.
Then $$\sh(C;x)=\max\bigl\{t \st\ [\pi^*C-t\ell]\in\Mov_1(\bl_xX)\bigr\}.$$
\end{prop}
\noindent Note that the maximum is nonnegative by the previous lemma, and well-defined (finite) because
for fixed ample $H$ on $\bl_xX$, we have $(\pi^*C-t\ell)\cdot H<0$ for $t\gg 0$.
\begin{proof}Let $t\geq 0$ with $\pi^*C-t\ell$ movable. Note that 
\begin{equation}\label{divmult}\overline L=\pi^*L-(\mult_xL)\cdot E\end{equation} is Cartier on $\bl_xX$ for any Cartier divisor $L$ on $X$. Then $(\pi^*C-t\ell)\cdot\overline L\geq 0$ for all effective Cartier divisors $L$
on $X$. By \eqref{multform} and the projection formula this is equivalent to $C\cdot L\geq t\cdot\mult_xL$. If $L$ ranges through effective Cartier divisors that pass through $x$, then we obtain $t\leq\sh(C;x)$. It follows that $\max\{t\geq 0\ |\ [\pi^*C-t\ell]\in\Mov_1(\bl_xX)\}\leq\sh(C;x)$.

Conversely, if $0\leq t\leq\frac{C\cdot L}{\mult_xL}$ for all effective Cartier $L$ through $x$, then $C\cdot L\geq t\cdot\mult_xL$.
This is also clearly true for $L$ effective Cartier not passing through $x$. As above, $(\pi^*C-t\ell)\cdot\overline L\geq 0$
for all effective Cartier $L$ on $X$. We also observe $(\pi^*C-t\ell)\cdot E=t\geq 0$. Since $x$ is a smooth point of $X$,
any effective Cartier divisor on $\bl_xX$ is the sum of an effective Cartier divisor of form $\overline L$ and a nonnegative
multiple of $E$. Using \cite{bdpp13} we conclude that $\pi^*C-t\ell$ is movable. Consequently $\pi^*C-\sh(C;x)\ell$ is movable.
\end{proof}

Seshadri constants for movable curves also verify a semi-continuity type statement analogous to the case
of divisors (\cite[Example 5.1.11]{laz04}).

\begin{prop}\label{prop:seshadrisemicontinuity}Let $T$ be a smooth variety over an uncountable algebraically closed field, and let $p:\mathscr X\to T$
be a smooth projective morphism with connected fibers. Assume that $p$ admits a section $x:T\to\mathscr X$. Let $\mathscr C\subset\mathscr X$ be a cycle of dimension $\dim T+1$.
Denote by $X_t$ the scheme theoretic fiber of $p$ over $t\in T$
and by $[C_t]$ the class of the restriction $\mathscr [C]|_{X_t}$ (in the sense of \cite[Chapter 8]{fulton84}).
Assume that $[C_t]\in\Mov_1(X_t)$ for all $t\in T$. Then $\sh([C_t];x_t)$ is constant for very general $t\in T$. For special $t$ it may only decrease.
\end{prop}

There are no effectivity assumptions on $\mathscr C$.
Recall that a property is said to hold for very general $t\in T$ if there exists an at most countable collection
of proper closed subsets $V_i\subsetneq T$ such that the property holds for all $t\in T\setminus\bigcup_iV_i$.

\begin{proof}Let $\pi:\widetilde{\mathscr X}\to\mathscr X$ be the blow-up of $\mathscr X$ along the image of $x$
with induced smooth morphism $q:\widetilde{\mathscr X}\to T$ and exceptional divisor $\mathscr E$. 
Let $\Lambda$ be a cycle on $\widetilde {\mathscr X}$ such that the class of its restriction to each $E_t$
is the same as the class of a line in the exceptional divisor of $\bl_{x_t}X_t$. Such cycles exist, even effective ones. For fixed $t_0$, apply the lemma below for $\pi^*\mathscr C-\sh([C_{t_0}];x_{t_0})\Lambda$ (here $\pi^*\mathscr C$ denotes a choice of a cycle representing the Chow class $\pi^*[\mathscr C]$ in the sense of \cite[Chapter 8]{fulton84}) to show that $\sh([C_t];x_t)\geq\sh([C_{t_0}];x_{t_0})$ for very general $t\in T$. Then apply it for $t_0$ very general to find that $\sh([C_t];x_t)$ is constant for $t$ very general.
\end{proof}

\begin{lem}\label{lem:movingjumps}Let $T$ be a smooth variety over an uncountable algebraically closed field, and let $p:\mathscr X\to T$ be a smooth projective morphism with connected fibers. 
Let $\mathscr C\subseteq\mathscr X$ a cycle of dimension $\dim T+1$.  
Assume there exists $t_0\in T$ such that $[C_{t_0}]\in\Mov_1(X_{t_0})$. 
Then $[C_t]$ is movable for very general $t\in T$. 
\end{lem}

\begin{proof}Let $(L_t,t)$ be the set of pairs with $t\in T$ and $L_t$ an irreducible divisor in $X_t$ with $[L_t]\cdot [C_t]<0$. By usual Hilbert scheme arguments, these are parameterized by countably many schemes $H_i$.
If the conclusion fails, then by the main result of \cite{bdpp13}, one of the $H_i$ dominates $T$.
Up to replacing $H_i$ by a closed subset, we may assume that the map $\epsilon:H_i\to T$ is generically
finite and dominant. 
Let $\mathscr D\subset\mathscr X$ be the closure of the union of the divisors $L_t$ parameterized
by $H_i$. It is an effective divisor that may contain some of the fibers of $p$ in its support. For general $t\in T$ though, the fiber $D_t$ is the sum $\sum_{\epsilon(h_i)=t}L_{h_i}$. In particular $[D_t]\cdot [C_t]<0$. 

Since $D$ is effective and $p$ is smooth, $[D_{t_0}]\in N^1(X_{t_0})$ is a pseudo-effective divisor class.
See \cite[Lemma 4.10]{fl16sw}.
For general $t\in T$ we find the contradiction $0\leq [D_{t_0}]\cdot [C_{t_0}]=[D_t]\cdot [C_t]<0$.
The intersection numbers are equal to $\mathscr D\cdot \mathscr C\cdot X_t$,
which makes sense after considering a smooth projective completion of $\mathscr X$.
\end{proof}

\begin{ex}[Toric varieties]\label{ex:toric}Let $X=X(\Delta)$ be a smooth complete toric variety.
Let $[C]\in\Mov_1(X)$. Let $x=x_{\sigma}$ be a torus invariant point.
Then the Seshadri constant is computed by one of the irreducible invariant divisors through $x_{\sigma}$. Specifically, $$\sh(C;x)=\min\bigl\{C\cdot D_{\tau}\st \tau\in\sigma(1)\bigr\}.$$
(A deformation argument shows that $\sh(C;x)=\min\bigl\{\frac{C\cdot D_{\tau}}{\mult_{x_{\sigma}}D_{\tau}}\st x_{\sigma}\in D_{\tau}\bigr\}$. All invariant divisors on a smooth toric variety are smooth.)
When $x$ is not a torus invariant point, the Sesadri constant is potentially bigger by the results above.

Note that when $C$ is effective, but not movable, there exists some $\tau$ such that $C\cdot D_{\tau}<0$. As in Remark \ref{rmk:nonmov}, we obtain $\sh(C;x)=-\infty$ for all $x\not\in D_{\tau}$. In particular the Seshadri constants are $-\infty$ on the dense torus $T$.

For an even more specific example, let $X$ be the blow-up of $\bb P^2$ at one point, and let $C=E$ be the exceptional divisor. The divisor $L=E$ is the only one that $C$ does not meet properly,
and we deduce that $\sh(C;x)=-1$ for all $x\in E$. By the arguments above, $\sh(C;x)=-\infty$ for all $x\in X\setminus E$. In particular $\sh(C;\cdot)$ may fail lower semi-continuity when $C$ is not movable.
\qed 
\end{ex}

\section{A Seshadri-type criterion}

\begin{proof}[Proof of Theorem A]One implication is provided by Lemma \ref{lem:uniformlowerbound}.
For the converse, assume first that $X$ is also smooth and $\varepsilon(C;x)\geq\varepsilon>0$ for some fixed $\varepsilon$ independent of $x\in X$. 
If $[C]$ is not an interior class, then by the duality result of \cite{bdpp13} (see \cite[Theorem 2.22]{fl13z} for the case of positive characteristic) there exists a pseudo-effective $\bb R$-Cartier $\bb R$-divisor $L$ with $[L]\neq 0$ such that $C\cdot L=0$. Let $L=P_{\sigma}(L)+N_{\sigma}(L)$
be the divisorial Zariski decomposition in the sense of \cite{nak} (or \cite{mustata}, \cite{chms14}, or \cite[Section 4]{fkl16} for the case of positive characteristic). We can assume that $[L]$ is extremal in $\Eff^1(X)$, and then either $L\equiv N_{\sigma}(L)$, or $L=P_{\sigma}(L)$.
If $L\equiv N_{\sigma}(L)$, then $L$ is effective (up to numerical equivalence). Choose $x$ in the support of some effective representative $L'$ of $[L]$. We obtain the contradiction $0=C\cdot L\geq\varepsilon\cdot\mult_xL'>0$.

We now treat the case when $L=P_{\sigma}(L)$ with $[L]\neq 0$.
By \cite[V.1.11. Theorem]{nak} (see \cite{chms14} for characteristic $p$), there exists $\beta>0$ and $A$ ample such that $$\dim_{\bb C}H^0\bigl(X,\cal O_X(\lfloor mL\rfloor+A)\bigr)\geq \beta\cdot m$$ 
for all $m$ sufficiently large. For any $D_m\in|\lfloor mL\rfloor+A|$ we have
$C\cdot(\lfloor mL\rfloor+A)\geq\varepsilon\mult_xD_m$ for any $x\in X$. Since $C\cdot L=0$, the left hand side is bounded independently of $m$. It follows that there exists some integer $B>0$ such that $B>\mult_xD_m$ for all $x\in X$, all sufficiently large $m$, and all $D_m\in|\lfloor mL\rfloor+A|$.
However, passing through a fixed smooth $x$ with multiplicity at least $B+1$ is a constant number of conditions
on any linear series on $X$. Since $|\lfloor mL\rfloor+A|$ has arbitrarily large dimension, for large $m$
we may find $D_m\in|\lfloor mL\rfloor+A|$ with $\mult_xD_m>B$. This is a contradiction.

\vskip.25cm
Consider now the case when $X$ is an arbitrary projective variety over $\bb C$. 
Let $\pi:\widetilde X\to X$ be a resolution of singularities of $X$. 
Let $H$ be a large ample on $X$. Since $\pi$ is birational, the divisor $\pi^*H$ is big,
so it can be written as $\pi^*H=A+E$ with $A$ ample on $\widetilde X$ and $E$ effective.
As before, there exists a pseudo-effective $\bb R$-Cartier $\bb R$-divisor $L$ on $X$
such that $C\cdot L=0$ and $[L]\neq 0$ in $N^1(X)$.

Assume that $P_{\sigma}(\pi^*L)$ is not numerically trivial. 
Up to replacing $H$ (so $A$ and $E$) by high multiples (see the proof of \cite[V.1.11. Theorem]{nak}), we may assume that $A$ and $P_{\sigma}(\pi^*L)$ are as in the smooth and movable case.
Fix $x$ in the smooth locus of $X$ such that $\pi$ is an isomorphism in a neighborhood of $x$, and denote $\tilde x=\pi^{-1}\{x\}$. Choose $D_m\in|\lfloor mP_{\sigma}(\pi^*L)\rfloor+A|$
such that $\lim_{m\to\infty}\mult_{\tilde x}D_m=\infty$. Note that
$$C\cdot H=C\cdot (mL+H)=C\cdot \pi_*\bigl(D_m+\langle mP_{\sigma}(\pi^*L)\rangle+mN_{\sigma}(\pi^*L)+E\bigr)\geq\varepsilon\cdot\mult_x\pi_*D_m=\varepsilon\cdot\mult_{\tilde x}D_m,$$
where by $\langle\cdot\rangle$ we denote the fractional part of a divisor, as in \cite[Section II.2.d]{nak}.
The second equality is true because the two divisors that we intersect with are linearly equivalent.
The inequality holds because $D_m+\langle mP_{\sigma}(\pi^*L)\rangle+mN_{\sigma}(\pi^*L)+E$
is a sum of effective $\bb R$-Weil divisors, and it stays so after pushforward. Furthermore the pushforward is an $\bb R$-Cartier $\bb R$-divisor linearly equivalent to $mL+H$. The last equality holds because
$\pi$ is an isomorphism above $x$. As $m$ grows we get a contradiction.

It remains to consider the case where $P_{\sigma}(\pi^*L)$ is numerically trivial. 
For all $m>0$ we have that $\frac 1mA+P_{\sigma}(\pi^*L)$ is an ample $\bb R$-divisor on $\widetilde X$,
in particular $\bb R$-linearly equivalent to an effective $\bb R$-divisor $F_m$. Since $[L]\neq 0$, it follows that $\pi_*N_{\sigma}(\pi^*L)$ is a nonzero effective Weil  $\bb R$-divisor on $X$. This is easily seen
by intersecting $L=\pi_*P_{\sigma}(\pi^*L)+\pi_*N_{\sigma}(\pi^*L)$ with $H^{n-1}$, where $n=\dim X$. Note that $L\cdot H^{n-1}>0$ since $[L]\neq 0$ is pseudo-effective, and $H^{n-1}$ is in the strict interior of $\Mov_1(X)$ (\cite[Lemma 3.9]{fl13z}).
Then 
$$\frac 1mC\cdot H=C\cdot\bigl(\frac 1mH+L\bigr)=C\cdot\pi_*\bigl(F_m+N_{\sigma}(\pi^*L)+\frac 1mE\bigr)\geq\varepsilon\cdot\mult_x\pi_*N_{\sigma}(\pi^*L)$$
for every $x$ in the support of $\pi_*N_{\sigma}(\pi^*L)$. As $m$ grows, we get a contradiction.

\vskip.25cm
Finally, we consider the case of fields of arbitrary characteristic. Instead of a resolution, consider 
$\pi:\widetilde X\to X$ a nonsingular alteration (\cite{dejong96}). 
For $H$ ample, $\pi^*H$ is big, so we can construct $A$ and $E$ as before. 
The proof goes through as above with minimal changes. In the projection formula there is a correction
by $\deg\pi$, e.g., $\deg\pi\cdot L=\pi_*\pi^*L$. The point $x$ is chosen in the regular locus of $X$
and in the finite locus of $\pi$, and $\tilde x$ is any point in $\pi^{-1}\{x\}$.
A bounding relation between $\mult_{\tilde x}D_m$ and $\mult_x\pi_*D_m$ is provided by
Lemma \ref{lem:multiplicitypush}.
\end{proof}

The following lemma provides some control for $\inf_{x\in X}\sh(C;x)$ under blow-ups of smooth varieties along smooth centers. We hope that it will inspire arguments following the steps of the MMP.

\begin{lem}Let $X$ be a smooth projective variety over $\bb C$, and let $C$ be a movable curve such that $\inf_{x\in X}\sh(C;x)\geq\varepsilon>0$. Let $Z\subset X$ be a smooth closed subvariety with ${\rm codim}_X\,(Z)\geq 2$,
and let $Y:=\bl_ZX$ with blow-down morphism $\pi$. Let $E$ be the exceptional divisor of $\pi$,
and let $\ell$ be a line in any fiber of $E\to Z$. Set $C':=\pi^*C-\frac{\varepsilon}2\ell$.
Then $C'$ is movable, $\pi_*C'=C$, and $\sh(C';y)\geq\frac{\varepsilon}2$ for all $y\in Y$.
Equality holds for $y\in E$.
\end{lem}
\begin{proof}It is clear that $\pi_*C'=C$. 
We want to prove the Seshadri inequality $C'\cdot L\geq\frac{\varepsilon}2\mult_yL$ 
for all irreducible divisors $L$ on $Y$ and for all $y\in Y$. If this holds, then in particular $C'\cdot L\geq 0$ for all effective divisors on $Y$, hence $C'$ is movable.
Consider first the case where $Z=\{z\}$ is a point. 
Let $y\in Y$ with $\pi(y)=:x\neq z$. Let $\overline D$ be an irreducible divisor on $Y$ containing $y$.
It follows that $\overline D\neq E$, hence it is the strict transform of a divisor $D$ on $X$.
We have $$C'\cdot\overline D=C\cdot D-\frac{\varepsilon}2\mult_zD\geq C\cdot D-\frac 12C\cdot D\geq\frac{\varepsilon}2\mult_xD=\frac{\varepsilon}2\mult_y\overline D,$$
since $\pi$ is an isomorphism around $x\neq z$.
Let now $y\in E$.
If $\overline D$ is the strict transform of an effective divisor $D$ on $X$, then 
$$C'\cdot\overline D=C\cdot D-\frac{\varepsilon}2\mult_zD\geq \bigl(\sh(C;x)-\frac{\varepsilon}2\bigr)\cdot\mult_zD\geq \frac{\varepsilon}2\mult_zD.$$
From Lemma \ref{lem:multiplicitypush}, we have $\mult_zD\geq\mult_y\overline D$. 
Furthermore $$C'\cdot E=\frac{\varepsilon}2=\frac{\varepsilon}2\mult_yE$$ for all $y\in E$, since $E$ is smooth.

In the general case, let $y\not\in E$ and let $\overline D$ be an irreducible divisor containing $y$. Set $x:=\pi(y)$. As above, $\overline D$ is the strict transform of an irreducible divisor $D$ on $X$. Furthermore
$$C'\cdot\overline D=C\cdot D-\frac{\varepsilon}2\mult_ZD= C\cdot D-\frac{\varepsilon}2\mult_zD\geq C\cdot D-\frac 12C\cdot D=\frac 12C\cdot D\geq\frac{\varepsilon}2\mult_xD=\frac{\varepsilon}2\mult_y\overline D,$$
since $\pi$ is an isomorphism around $x$. Here $z$ is a general point of $Z$ so that $\mult_ZD=\mult_zD$. 

Let now $y\in E$ so $x:=\pi(y)\in Z$. If $\overline D$ is the strict transform of an effective divisor $D$ on $X$, then
$$C'\cdot\overline D=C\cdot D-\frac{\varepsilon}2\mult_ZD\geq \frac{\varepsilon}2\cdot\bigl(2\mult_x D-\mult_ZD\bigr).$$
We show that 
\begin{equation}\label{eq:multstrict}2\mult_xD-\mult_ZD\geq\mult_y\overline D\end{equation} 
This is a local computation.
In a small analytic neighborhood of $y$ we have that $\pi$ is given by 
$$(y_1,y_2,\ldots,y_n)\mapsto(y_1,y_1y_2,\ldots,y_1y_d,y_{d+1},\ldots,y_n),$$ where $d$ is the codimension of $Z$, and $Z$ is given by $x_1=\ldots=x_d=0$. A local equation for $D$ around $x$ is 
$$f(x_1,\ldots,x_n)=f_m(x_1,\ldots,x_n)+f_{m+1}(x_1,\ldots,x_n)+\cdots,$$
where $f_i$ is homogeneous of degree $i$ and $m:=\mult_xD$. Then a local equation for
$\overline D$ around $y$ is 
$$\frac{f_m(y_1,y_1y_2,\ldots,y_1y_d,y_{d+1},\ldots,y_n)}{y_1^{\mult_ZD}}+\frac{f_{m+1}(y_1,y_1y_2,\ldots,y_1y_d,y_{d+1},\ldots,y_n)}{y_1^{\mult_ZD}}+\cdots$$
If a monomial $x_1^{a_1}\cdot\ldots\cdot x_n^{a_n}$ appears in the equation of $D$,
then $y_1^{a_1+\ldots+a_d-\mult_ZD}\cdot y_2^{a_2}\cdot\ldots\cdot y_n^{a_n}$
appears in the equation of $\overline D$. The correspondence is reversible, hence one-to-one. 
The degree is $a_1+2(a_2+\ldots+a_d)+a_{d+1}+\ldots+a_n-\mult_ZD$. Since $f_m\neq 0$, at least one of these has degree at most $2m-\mult_ZD$ as desired. 

It remains to verify the Seshadri inequality for the divisor $E$ and a point $y\in E$. Using that $E$ is smooth, 
$$C'\cdot E=\frac{\varepsilon}2=\frac{\varepsilon}2\mult_yE.$$
\end{proof}

\begin{rmk}It is tempting to approach \eqref{eq:multstrict} by proving $\mult_xD\geq\mult_y\overline D$ and then using $\mult_xD\geq\mult_ZD$. 
While the first inequality is true if $Z$ is a point, unfortunately blow-ups may increase multiplicity when the dimension of $Z$ is positive.
Consider in $\bb A^3_{\bb C}$ with coordinates $(x_1,x_2,x_3)$ the blow-up of the line $Z:=V(x_1,x_2)$
and the divisor $D$ of equation $x_2+x_3^3=0$. 
It has multiplicity $1$ at $x=(0,0,0)$, and multiplicity $0$ along $Z$.
Its strict transform $\overline D$ agrees with the pullback, and has equation $y_1y_2+y_3^3=0$.
It has multiplicity $2$ at $y=(0,0,0)$.  
\end{rmk}

There is also some control on pushforwards.

\begin{lem}Let $\pi:X\to Y$ be a dominant morphism of projective varieties over an algebraically closed field. Let $C\in\Mov_1(X)$. Then $\sh(\pi_*C;\pi(x))\geq\sh(C;x)$ for all $x\in X$ such that $x$ and $\pi(x)$ are smooth.
\end{lem}
\begin{proof}Let $D$ be an effective Cartier divisor through $y:=\pi(x)$.
It is an immediate local computation that $\mult_{y}D\leq\mult_x\pi^*D$.
Then $\frac{\pi_*C\cdot D}{\mult_yD}\geq\frac{C\cdot\pi^*D}{\mult_x\pi^*D}\geq\sh(C;x)$.
\end{proof}

We can also bound the difference  
$\sh(\pi_*C;\pi(x))-\sh(C;x)$ from above for smooth blow-ups.

\begin{lem}Let $X$ be a complex projective manifold, and let $Z\subset X$ be a smooth closed 
subvariety with ${\rm codim}_X(Z)\geq 2$.
Its blow-up $\pi:\bl_ZX\to X$ has exceptional divisor $E$.
Let $x\in X$ and $y\in\pi^{-1}x$ be closed points.
Consider $C\in\Mov_1(\bl_ZX)$.
Set $R_y:=\liminf_{\delta\to 0^+}\frac{\mult_Z\pi_*\overline D_{\delta}}{\mult_x\pi_*\overline D_{\delta}}$, where $\overline D_{\delta}$ ranges through the set of irreducible divisors through $y$ such that
$\frac{C\cdot\overline D_{\delta}}{\mult_y\overline D_{\delta}}<\sh(C;y)+\delta$. 
Then 
$$\sh(C;y)+C\cdotp E \geq 2\cdotp\sh(C;y)+(C\cdotp E-\sh(C;y))\cdotp R_y\geq  \sh(\pi_*C;x) \geq \sh(C;y)$$
if $x\in Z$ and $C\cdotp E>\sh(C;y)$, and
$$\sh(C;y)+(C\cdotp E)\cdotp R_y \geq  \sh(\pi_*C;x) \geq  \sh(C;y),$$
if $x\not\in Z$. 
\end{lem}

\begin{proof}Only the upper bounds on $\sh(\pi_*C;x)$ need further justification.
For $\overline D_{\delta}$ as above, denote $D_{\delta}:=\pi_*\overline D_{\delta}$, such that $\pi^*D_{\delta}=\overline D_{\delta}+\mult_ZD_{\delta}\cdot E$. When $y\in E$, the assumption $C\cdotp E>\sh(C;y)$ shows that $\overline D_{\delta}\neq E$ for $\delta$ sufficiently small. 
In particular $D_{\delta}\neq 0$ and $\mult_xD_{\delta}>0$.
$$\sh(\pi_*C;x)\leq\frac{\pi_*C\cdotp D_{\delta}}{\mult_xD_{\delta}}=\frac{C\cdotp\overline D_{\delta}+\mult_ZD_{\delta}\cdot(C\cdotp E)}{\mult_xD_{\delta}}<\frac{\mult_y\overline D_{\delta}\cdot(\sh(C;y)+\delta)+\mult_ZD_{\delta}\cdot(C\cdotp E)}{\mult_xD_{\delta}}.$$
If $x\in Z$, then using \eqref{eq:multstrict}, further upper bounds are
$$\frac{2\mult_xD_{\delta}\cdot(\sh(C;y)+\delta)+\mult_ZD_{\delta}\cdot(C\cdotp E-\sh(C;y)-\delta)}{\mult_xD_{\delta}}\leq 2\cdot(\sh(C;y)+\delta)+(C\cdotp E-\sh(C;y)).$$
We get the desired upper bounds by letting $\delta$ tend to $0$.
If $x\not\in Z$, then $x$ is identified with $y$, and 
$\mult_y\overline D_{\delta}=\mult_xD_{\delta}$.
Again we get the desired upper bound by letting $\delta$ tend to $0$.
\end{proof}

We exhibit a curve class $C$ where 
$\sh(C;y)+(C\cdotp E)\cdotp R_y =  \sh(\pi_*C;x) >  \sh(C;y)$
for general $y$. This shows that the previous lemma is sharp.

\begin{ex}Linear projection from $x\in G(2,4)\subset\bb P^5$ gives a birational map $p:G(2,4)\dashrightarrow\bb P^4$ defined outside $x$.
It contracts all lines in $G(2,4)$ that pass through $x$, and it can be resolved by blowing-up $x$.
We obtain a diagram
$$\xymatrix{\bl_xG(2,4)\ar[d]_{\sigma}\ar@{=}[r]&\bl_Q\bb P^4\ar[d]^{\pi} \\
G(2,4)\ar@{-->}[r]_p& \bb P^4}
$$
where $Q=\bb P^4\cap\bb T_xG(2,4)\cap G(2,4)$ is a quadric surface.
Denote by $h$ the class of a linear hyperplane in $\bb P^4$ and by $\xi$ the class of a Pl\" ucker hyperplane on $G(2,4)\subset\bb P^5$.
Denote by $\Lambda=\frac 12\xi^3$ the class of a line in $G(2,4)$, by $L=h^3$ the class of a line in $\bb P^4$, 
by $\ell$ the class of a line in the exceptional divisor $E\simeq\bb P^3$ of $\sigma$,
and by $f$ the class of a line inside the exceptional divisor $F\simeq\bb P_Q\bigl(N^{\vee}_Q\bb P^{4}\bigr)$ of $\pi$.

We have relations $\pi^*h=\sigma^*\xi-E$ and $F=\sigma^*\xi-2E$. Similarly $\pi^*L=2\sigma^*\Lambda-\ell$ and $f=\sigma^*\Lambda-\ell$. 
From these, we deduce $\pi_*\sigma^*\Lambda=L$ and $\sigma_*\pi^*L=2\Lambda$.

We have $\sh(\sigma^*\Lambda;y)=\left\{\begin{array}{cc}\frac 12&\mbox{, if }y\not\in E\\ 0&\mbox{, if }y\in E\end{array}\right.$ (see Example \ref{G(2,4)}),
and $\sh(\pi^*L;y)=\left\{\begin{array}{cc}1&\mbox{, if }y\not\in F\\ 0&\mbox{, if }y\in F\end{array}\right.$.

We see that $\sh(\pi_*\sigma^*\Lambda;\pi(y))>\sh(\sigma^*\Lambda;y)$ for all $y\not\in F$.
We have $\sigma^*\Lambda\cdot F=1$, hence $R_y\geq\frac 12$. 
In fact equality holds for $y\not\in E\cup F$. 
The Seshadri constant of $\sigma^*\Lambda$ at $y$ is computed by $\overline D_{\delta}$, the strict transform on $\bl_Q\bb P^4=\bl_xG(2,4)$ of 
$\bb T_{\sigma(y)}G(2,4)\cap G(2,4)$.
It is easy to compute that $\mult_Q\pi(\overline D_{\delta})=1$ and $\mult_{\pi(y)}\pi(\overline D_{\delta})=2$.\qed
\end{ex}

\section{Null locus}

\begin{thrm}[A characterization of the Null locus]
\label{thm:nulllocus}Let $X$ be \underline{\emph{smooth}} projective over an algebraically closed field. 
Let $C\in\Mov_1(X)$ and assume there exists $x_0\in X$ such that $\sh(C;x_0)>0$. 
\begin{enumerate}[i)]
\item If $[L]\in\Eff^1(X)$ satisfies $C\cdot L=0$, then $L\equiv N_{\sigma}(L)$, where $N_{\sigma}(L)$ denotes the negative part in the divisorial Zariski decomposition of \cite{nak}.
Furthermore the cone $$C^{\perp}\cap\Eff^1(X):=\bigl\{[L]\in\Eff^1(X)\st C\cdot L=0\bigr\}$$ is simplicial, generated by the classes of 
finitely many effective irreducible divisors $L_1,\ldots,L_r$ that do not pass through $x_0$. 
As a cycle, any effective $\bb R$-Cartier $\bb R$-divisor $L$ whose class is in this cone is
necessarily a nonnegative linear combination of the $L_i$.
\item The class $C$ is \emph{big}, that is in the strict interior of the Mori cone of curves $\Eff_1(X)\subset N_1(X)$.
In fact it is in the strict interior of the possibly smaller dual cone $\Mov^1(X)^{\vee}$. \footnote{Recall that the \emph{movable cone of divisors} $\Mov^1(X)\subset N^1(X)$ is the closure of the convex cone generated by classes of Cartier divisors in linear series without fixed divisorial components.}
\item The ``null locus'' ${\rm Null}\, (C):=L_1\cup\ldots\cup L_r$ coincides with the set
$\bigl\{x\in X\st \sh(C;x)=0\bigr\}$. In particular the latter is Zariski closed in $X$.
\end{enumerate}
\end{thrm}

\begin{proof}$i)$. Note that $\sh(C;x_0)>0$ implies $C\neq 0$.
If $[L]\in\Eff^1(X)$ verifies $C\cdot L=0$, then $C\cdot P_{\sigma}(L)=0$.
If $P_{\sigma}(L)\not\equiv 0$, then the multiplicity arguments in the proof of Theorem A
applied to the point $x_0$ contradict $\sh(C;x_0)>0$. We conclude that $L\equiv N_{\sigma}(L)$. In particular $L$ is numerically equivalent to an effective divisor.
For any effective $L$ with $[L]$ in $C^{\perp}\cap\Eff^1(X)$, the irreducible components of $\Supp L$ are also in the cone,
and \cite[III.1.10. Proposition]{nak} proves that the numerical classes of these components are linearly independent
in $N^1(X)$.
Furthermore, no component of $\Supp L$ may pass through $x_0$ because $\sh(C;x_0)>0$.
Choose finitely many effective divisors $G_j$ whose classes span
the subspace $V$ generated by $C^{\perp}\cap\Eff^1(X)$ in $N^1(X)$.
Let $L_i$ be the finite set of irreducible divisors that appear as components of any of the $G_j$.
Clearly $[L_i]$ also generate $V$.
By looking at $\sum_jG_j$, we see that the irreducible divisors above have classes that are linearly independent in $N^1(X)$. In particular these classes form a basis of $V$. If $E$ is any effective divisor with $[E]\in C^{\perp}\cap\Eff^1(X)$,
let $E_0$ denote any of the irreducible components of its support.
By looking at $\sum_iL_i+E_0$, we deduce that $E_0$ is one of the $L_i$, or else 
$L_1,\ldots,L_r,\ E_0$ also have linearly independent classes.

$ii)$. If $C$ is not big or in the strict interior of $\Mov^1(X)^{\vee}$, 
then there exists some nef/movable divisor $L$ with $C\cdot L=0$ and $[L]\neq 0$.
Since nef divisors are movable, in both cases we have $L=P_{\sigma}(L)$.
Using $i)$, we find the contradiction $[L]=0$.

$iii)$. If $x\in L_i$ for some $i$, then $\sh(C;x)=0$ since $C\cdot L_i=0$ and $\mult_xL_i>0$.
Conversely, assume $\sh(C;x)=0$. Then there exists a sequence $D_j$ of effective divisors with
irreducible supports such that $\lim_{j\to\infty}\frac{C\cdot D_j}{\mult_xD_j}=0$ and $x\in\Supp D_j$ for all $j$.
Up to rescaling, we may assume that $H^{n-1}\cdot D_j=1$ for some very ample divisor $H$ on $X$, where $n:=\dim X$.
By B\' ezout, $\mult_xD_j\leq H^{n-1}\cdot D_j=1$. By \cite[Theorem 1.4.(3)]{flpos}, the sequence $[D_j]\in N^1(X)$ is bounded, so up to passing to a subsequence, we may assume that $\lim_{j\to\infty}[D_j]=[D]$ for some $D\in\Eff^1(X)$. Furthermore $H^{n-1}\cdot D=1$, so $[D]\neq 0$. From the bound on the multiplicity of $D_j$, we also deduce $C\cdot D=0$.

By part $i)$, we see that $D\not\equiv P_{\sigma}(D)$. Thus there exists some irreducible divisor $E$ (in fact one of the $L_i$) on $X$ with associated
valuation $\sigma_E$ such that $\sigma_E(D)>0$. By the lower semi-continuity of $\sigma_E$ (cf. \cite[III 1.7.(1) Lemma]{nak}), we have $\sigma_E(D_j)>0$ for large $j$ (after maybe passing to a subsequence). Since $\Supp D_j$ is irreducible, it follows that $\Supp D_j=E$. Using $D_j\cdot H^{n-1}=1$ and the irreducibility of $E$, we find that $D_j$ is an eventually constant sequence, so $[D]=[D_j]$ for large $j$. 
By part $i)$ we deduce that $D_j$ is a nonnegative linear combination of $L_i$
and in particular that $x\in \Supp D_j$ is contained in ${\rm Null}\, (C)$.
\end{proof}

\begin{rmk}One can also use the work of \cite{lxpos} as explained in Remark \ref{rmk:lx}
to prove that if $\sh(C;x_0)>0$, and $C\cdot L=0$ for $[L]\in\Eff^1(X)$, 
then $L\equiv N_{\sigma}(L)$ and $[C]$ is in the strict interior of $\Mov^1(X)^{\vee}$, in particular it is big.

Moreover, \cite{lxpos} prove that there exists a movable divisor $P$ (unique up to numerical equivalence) such that $C$
agrees with the positive product $\langle P^{n-1}\rangle$ of \cite{bfj}.
Then \cite{mx17} show that  
${\rm Null}\, (C)={\bf B}_+(P)$.   
\end{rmk}

\begin{ex}\label{ex:POL}Let $X$ be a smooth projective variety, and let $C$ be a smooth curve on $X$
with ample normal bundle. Then $[C]$ is big.

(By \cite[Theorem 8.4.1]{laz04} and \cite{bdpp13}, we know that $C$ is movable.
In fact $C\cdot D>0$ whenever $D$ is a non-zero effective divisor that meets $C$.
Let $\pi:\bl_CX\to X$ be the blow-up of $C$ on $X$
with exceptional divisor $E$. Let $\ell\subset E$ be a line in any of the fibers over $C$. 
Let $N:=N_CX$ be the normal bundle of $C$, which is ample by assumption. 
We claim that $\pi^*C-\epsilon\ell$ is movable for sufficiently small $\epsilon>0$, and $\sh(\pi^*C-\epsilon\ell;y)>0$ for all $y\in E$. If true, it follows that $\pi^*C-\epsilon\ell$ is big, hence so is its pushforward $C$.
$$(\pi^*C-\epsilon\ell)\cdot E=\epsilon>0.$$
If $\overline F\subset\bl_CX$ is an effective divisor that does not meet $C$, then $(\pi^*C-\epsilon\ell)\cdot\overline F=0$. 
Let $\overline D\neq E$ be an irreducible effective divisor through some $y\in E$, 
and let $D:=\pi_*\overline D\neq 0$. 
$$\frac{(\pi^*C-\epsilon\ell)\cdot\overline D}{\mult y\overline D}\geq \frac{(\pi^*C-\epsilon\ell)\cdot\overline D}{\mult y\overline D\cap E}=\frac{\bigl(\jmath_*(\xi^{n-2}+(c_1(N)-\epsilon)\xi^{n-3}f)\bigr)\cdot\overline D}{\mult y\overline D\cap E}=$$ 
$$=\frac{(\xi^{n-2}+(c_1(N)-\epsilon)\xi^{n-3}f)\cdot(\overline D\cap E)}{\mult_y\overline D\cap E},$$
where $\jmath:E\to\bl_CX$ is the inclusion map, 
where $\xi$ is the class of $\cal O_{\bb P(N^{\vee})}(1)$,
and $f$ is the class of a fiber of $g$.
We have used the blow-up formula \cite[Proposition 6.7.(a)]{fulton84}.
To conclude, by Lemma \ref{lem:uniformlowerbound} it is enough to prove that 
$\xi^{n-2}+c_1(N)\xi^{n-3}f$ is in the strict interior of
$\Mov_1(\bb P_C(N^{\vee}))$. 
This can be verified directly by \cite{ful11}. 
)\qed
\end{ex}

The preivous example has been conjectured by Peternell. Ottem \cite{Ottem16} and Lau \cite{Lau17} have found different proofs. 
We see the relation between bigness and the positivity of Seshadri constants again for a special class of curves.

\begin{defn}[\cite{voisin}]Let $X$ be a smooth projective variety, and let $V\subseteq X$
be a subvariety of dimension $k$. Say that $V$ is \emph{very moving} if for a very general
$x\in X$ we have that for a very general $k$-dimensional subspace $W\subseteq T_xX$ there
exists a deformation $V'$ of $V$ passing through $x$ with $V'$ smooth and $T_xV'=W$. 
\end{defn}

General complete intersections of very ample divisors are natural examples of very moving subvarieties.

\begin{remark} Note that if $C$ is a very moving curve, then $\sh(C;x)\geq 1$ for very general $x\in X$.
The theorem implies then that $[C]$ is in the strict interior of $\Mov^1(X)^{\vee}$.
\cite[Proposition 2.7]{voisin} proves that very moving curves are in the strict interior of $\Eff_1(X)$.
\cite{voisin} also conjectures that very moving subvarieties of arbitrary dimension $k$ have classes in the
strict interior of $\Eff_k(X)$ and shows that this implies the generalized Grothendieck--Hodge conjecture for complete intersections of coniveau 2 in projective spaces. 
\end{remark}

\section{Jet separation}\label{s:jet}

\begin{rmk}\label{verymoving}Let $x\in X$ be a \emph{smooth} point. If $C$ is a curve such that for all irreducible divisors $D$ through $x$ 
we have that some deformation of $C$ meets $D$ properly and passes through $x$,
then $\sh(C;x)\geq 1$. (For any effective Cartier divisor $L$ through $x$, we can find a deformation $C'$ of $C$ that meets $L$ properly and also passes through $x$. Then $C\cdot L\geq\mult_xC'\cdot\mult_xL\geq\mult_xL$ by \cite[Theorem 12.4]{fulton84}.)  

In particular, when $X$ is smooth and the above condition holds for all $x\in X$, such curves $C$ have classes in the strict interior of $\Mov_1(X)$ by Theorem A. 
\end{rmk}

We see the above as a counterpart to the statement that $\sh(L;x)\geq 1$ if $L$ is ample and 
globally generated (\cite[Example 5.1.18]{laz04}). 

\begin{ex}\label{G(2,4)}
For $X$ be the $4$-dimensional Grassmann variety $G(2,4)$, and let $\ell$ the class of a line. Then $\sh(\ell;x)=\frac 12$ for all $x\in X$.
(Let $x\in X$ and let $D_x:=X\cap T_xX$, the intersection 
taking place in $\bb P^5$. Then $\mult_xD_x=2$ and $\frac{\ell\cdot D_x}{\mult_xD_x}=\frac 12$.
This gives the bound $\sh(\ell;x)\leq\frac 12$. 
Since $\sh(\ell;x)<1$, the line $\ell$ cannot satisfy the conditions of Remark \ref{verymoving}.
This can also be checked directly. 
Any line on $X$ through $x$ is contained in $T_xX$, hence also in $D_x$
and cannot be moved to meet $D_x$ properly through $x$.
However, $2\ell$ is the class of the complete intersection of 3 general members of $|L|$, and so it
does satisfy the conditions of Remark \ref{verymoving}. We deduce the reverse inequality $\sh(2\ell;x)\geq 1$. ) 
\end{ex}

\begin{defn}[``Jet separation'' for curves]Let $X$ be a projective variety, and let $x\in X$. Let $[C]\in\Mov_1(X)$ be a $\bb Z$-class.
Denote by $s(C;x)$ the largest nonnegative integer $s$ for which there exists an integer $N\geq 1$ 
such that for every effective divisor $L$ through $x$ there exists an effective $\bb Z$-cycle $C'\equiv N\cdot C$
with $\mult_xC'\geq Ns$, and $C'$ meets $L$ properly.
When no such $s$ exists, set $s(C;x)=-1$.
\end{defn}

\begin{remark}
By asking that $C'$ is merely a $\bb Q$-cycle, one can leave out the integer $N$. It is a restriction on the denominators of $C'$, slightly weaker than
the condition $L\in|L|$ in the case of divisors.
\end{remark}

\begin{rmk}If $C$ is movable and $L$ is an effective Cartier divisor through a \emph{smooth} point $x$, then $C\cdot L\geq\frac 1N\mult_xC'\cdot\mult_xL$ for any effective $\bb R$-cycle $C'$ with $C'\equiv N\cdot C$ whose support meets $L$ properly.
It follows that $\sh(C;x)\geq s(C;x)$. By passing to multiples, $\sh(C;x)\geq\sup_k\frac{s(kC;x)}{k}$.
\end{rmk}

\begin{rmk}If $C_1$ and $C_2$ are movable $\bb Z$-classes with $s(C_1;x)\geq 0$ and $s(C_2;x)\geq 0$, then $s(C_1+C_2;x)\geq s(C_1;x)+s(C_2;x)$. (If integers $N_1,N_2$ are as in the definition for $s(C_1;x)$ and $s(C_2;x)$ respectively, then ${\rm lcm}\,(N_1,N_2)$ satisfies the condition in the definition for $s=s(C_1;x)+s(C_2;x)$.) In particular $\sup_k\frac{s(kC;x)}{k}=\limsup_{k\to\infty}\frac{s(kC;x)}{k}=\lim_{k\to\infty}\frac{s(kC;x)}k$.
\end{rmk}

\begin{thrm}\label{thrm:asymptoticjetseparation}Let $X$ be a projective variety over an algebraically closed field.
Let $x\in X$ be a smooth point, and let $[C]$ be an integral class in the strict interior of 
$\Mov_1(X)$. Then $$\sh(C;x)=\lim_{k\to\infty}\frac{s(kC;x)}k.$$
\end{thrm}

\begin{proof}It is enough to prove that $\liminf_{k\to\infty}\frac{s(kC;x)}k\geq\sh(C;x)$.
Since $[C]$ is in the strict interior of $\Mov_1(X)$, we have that $\sh(C;x)>0$.
Let $0<\frac pq<\sh(C;x)$ be a rational approximation of $\sh(C;x)$.
Lemma \ref{lem:strictpospull} below gives that $[q\pi^*C- p\ell]$ is in the strict interior of $\Mov_1(\bl_xX)$. The version of the main result of \cite{bdpp13} in \cite[Theorem 11.4.19]{laz04}
implies that there exist finitely many birational morphisms $Y_i\to X$ and $Z_j\to\bl_xX$
such that $[C]$ and respectively $[q\pi^*C-p\ell]$ are $\bb Q_{+}$-linear combinations of pushforwards of complete intersection curves (of very ample divisors) from $Y_i$ and $Z_j$ respectively.
We use here that $[C]$ and $[q\pi^*C-p\ell]$ are in the strict interior of $\Mov_1(X)$ and $\Mov_1(\bl_xX)$ respectively.

Choose $N$ large enough to clear all denominators in the $\bb Q_{+}$-linear combinations above.
For any effective divisor $L$ through $x$ choose effective divisors $L_i$ on each of $Y_i$ and $L_j$ on each $Z_j$ that surject onto it. 
Find general complete intersections on $Y_i$ that meet $L_i$ properly and similarly for $Z_j$.
Pushing from $Y_i$ to $X$ and adding up we find an effective curve of class $N\cdot [C]$ that meets $L$ properly. Pushing from $Z_j$ to $\bl_xX$ and adding up, we find an effective curve of class $N\cdot[q\pi^*C-p\ell]$. Its pushforward to $X$ is an effective curve of class $Nq\cdot[C]$ with multiplicity at least $Np$ at $x$. It also meets $L$ properly.

For any positive integer $k$, write $k=mq+q_1$ with $0\leq q_1<q$. From the discussion above we construct an effective $\bb Z$ 1-cycle of class $Nk\cdot [C]$ with multiplicity at least $Nmp$ at $x$ that meets $L$ properly. Then $\frac{s(kC;x)}{k}\geq \frac{mp}{k}=\frac{mp}{mq+q_1}$.
As $k$ grows, the latter approximates $\frac pq$. We conclude by letting $\frac pq$ tend to $\sh(C;x)$.
\end{proof}

\begin{lem}\label{lem:strictpospull}Let $x\in X$ be a smooth point and let $[C]$ be a class in the strict interior of $\Mov_1(X)$.
With the usual notations for the blow-up of $x$, we have that $\pi^*C-t\ell$ is in the strict interior
of $\Mov_1(\bl_xX)$ for all $0<t<\sh(C;x)$.
\end{lem}

\begin{proof}Note that $\sh(C;x)>0$ since $[C]$ is in the strict interior of $\Mov_1(X)$.
If for some $t\in(0,\sh(C;x))$ the class $[\pi^*C-t\ell]$ is not in the strict interior of $\Mov_1(\bl_xX)$,
there exists a pseudo-effective $\bb R$-Cartier $\bb R$-divisor $\overline D$ with $[\overline D]\neq 0$
on $\bl_xX$ and with $(\pi^*C-t\ell)\cdot \overline D=0$. Since $(\pi^*C-\sh(C;x)\ell)\cdot \overline D\geq 0$
and $\pi^*C$ is movable so $\pi^*C\cdot \overline D\geq 0$, using $t<\sh(C;x)$ we find
$\pi^*C\cdot\overline D=\ell\cdot\overline D=0$. Set $D:=\pi_*\overline D$. 
It is still an $\bb R$-Cartier $\bb R$-divisor, since $x$ is smooth. It is pseudo-effective, and $C\cdot D=0$ by the projection formula. Since
$[C]$ is in the strict interior of $\Mov_1(X)$, necessarily $[D]=0$.
For example by \cite[Theorem 4.13]{fl16sw}, it follows that $[\overline D]$ is effective.
The only effective divisors on $\bl_xX$ with numerically trivial pushforward via $\pi$ are the multiples of the exceptional $E$. From $E\cdot l=-1$ we deduce $[\overline D]=0$, which contradicts our choice.
\end{proof}

An easier asymptotics-free result holds for $\bb R$-classes.

\begin{cor}[Theorem C] Let $x\in X$ be a smooth point and let $[C]$ be an $\bb R$-class in the strict interior of $\Mov_1(X)$. Then $\sh(C;x)$ is the supremum of all $s\geq 0$ such 
that for all effective $\bb R$-Weil $\bb R$-divisors $L$ with $x\in\Supp L$
there exists an effective $\bb R$-1-cycle $C'\equiv C$ such that $\Supp C'$ and $\Supp L$ meet properly and $\mult_xC'\geq s$. The same statement holds with $\bb Q$ replacing $\bb R$ throughout.
\end{cor}

\begin{proof}The B\' ezout arguments of Remark \ref{verymoving} prove that
$\sh(C;x)\geq s$ for all such $s$. For the reverse inequality, let $0<t<\sh(C;x)$.
By Lemma \ref{lem:strictpospull}, the class of $\pi^*C-t\ell$ is in the strict interior of $\Mov_1(\bl_xX)$, by \cite{bdpp13} (the version in \cite[Theorem 11.4.19]{laz04}) it is represented by some $\bb R_+$-combination of complete intersection curves. Its pushforward represents $[C]$, has multiplicity $t$ at $x$,
and by genericity can be chosen to meet any given divisor on $X$ properly.  
This proves that $\sh(C;x)$ is bounded above by the supremum of all $s$
as in the statement. The same arguments hold for rational coefficients.
\end{proof}

\section{Bounds on Seshadri constants}

\subsection{Lower bounds}
\begin{lem}Let $X$ be a projective variety of dimension $n\geq 2$ over $\bb C$. Let $H$ be an ample $\bb R$-divisor. Then 
$\sh(H^{n-1};x)\geq \sh(H;x)^{n-1}$ for all $x\in X$.
\end{lem}

\begin{proof}Replace $\frac 1m$ in the proof of Lemma \ref{lem:uniformlowerbound} by $\sh(H;x)$.
\end{proof}

\begin{cor}Let $H$ be an ample $\bb Z$-divisor on a smooth projective variety $X$ of dimension $n$ over $\bb C$.
Then $\sh(H^{n-1};x)\geq \frac 1{n^{n-1}}$ for very general $x\in X$.
\end{cor}
\begin{proof}Use the bound $\sh(H;x)\geq\frac 1n$ of \cite[Theorem 1]{ekl95}.
\end{proof}

\begin{prop}\label{prop:lowerboundci}Let $X$ be a projective variety of dimension $n$ over $\bb C$. Let $H$ be an ample
Cartier $\bb Z$-divisor on $X$ such that $(H^n)\geq n^{n-2}$, or $\sh(H;x_0)\geq 1$ for some $x_0\in X$, or more generally 
\begin{equation}\label{seshadribadcondition}\sh(H;x_0)^{n-2}\cdot(H^n)\geq 1.\end{equation} Then $\sh(H^{n-1};x)\geq 1$ for very general $x\in X$.
In fact we can choose $x$ outside the union of the singular locus of $X$ with countably many 
closed subsets of codimension two or more.
\end{prop}

The artificial condition \eqref{seshadribadcondition} is automatically satisfied when $X$ is a surface.

\begin{proof}The proof mimics the surface case from \cite{el93}. If the result fails, then by usual Hilbert scheme arguments one can find a variety $T$ and a family of irreducible divisors on $X$ denoted $\cal L\subseteq T\times X$, flat over $T$, with a section $T\to\cal L:t\mapsto (t,x_t)$ such that $H^{n-1}\cdot L_t<\mult_{x_t}L_t$ for all $t\in T$ and $\bigcup_{t\in T}x_t\subseteq X$ is a constructible set whose closure has dimension at least $n-1$ and meets the smooth locus of $X$. 
Since $H^{n-1}\cdot L_t<\mult_{x_t}L_t$, we deduce $\mult_{x_t}L_t\geq 2$, so that $x_t$ is a singular point of $L_t$. Each $L_t$ is reduced, hence generically smooth. We conclude that $L_t$
is not a constant divisor, so in particular the family $\cal L$ covers $X$. 
If we let $m:=\mult_{x_t}L_t$ for general $t\in T$, then we can guarantee that 
an infinitesimal deformation $L'_t$ of $L_t$ has multiplicity at least $m-1$ at $x_t$ and meets $L_t$ properly.

For any $k>0$ denote by $s_k$ the largest integer $s$ such that the linear series $|kH|$ separates $s$-jets at $x_t$. By choosing $n-2$ elements $D_1,\ldots,D_{n-2}$ in $|kH|$ general among those with multiplicity at least $s_k$ at $x_t$, we then have
$$m(m-1)\cdot\left(\frac{s_k}k\right)^{n-2}\cdot (H^n)\leq (L_t^2\cdot H^{n-2})\cdot(H^n)\leq(L_t\cdot H^{n-1})^2\leq (m-1)^2.$$
The first inequality is because $L_t^2\cdot (kH)^{n-2}=L_t\cdot L'_t\cdot D_1\cdot\ldots\cdot D_{n-2}$ and we apply \cite[Example 12.4.9]{fulton84} to the infinitesimal situation of $n$ divisors meeting properly with multiplicities at least $m, m-1, s_k,\ldots,s_k$ respectively at $x_t$. 
The second is a Hodge inequality on a surface obtained as complete intersection of members of $|kH|$.
The last inequality is because $L_t\cdot H^{n-1}<m$ and $L_t\cdot H^{n-1}$ is an integer.  

By \cite[Theorem 5.1.17]{laz04}, we have $\lim_{k\to\infty}\frac{s_k}k=\sh(H;x)$.
After taking limits, the assumption \eqref{seshadribadcondition} leads to the contradiction $2\leq m(m-1)\leq (m-1)^2$.
\end{proof}

\begin{remark}Example \ref{ex:grass} shows that if $C$ is a $\bb Z$ 1-cycle with class in the strict interior of $\Mov_1(X)$, then the inequality $\sh(C;x)\geq 1$ may fail for all $x\in X$. Also we do not have an universal lower bound
on $\sh(C;x)$ for very general points $x\in X$ independent of $\dim X$, as is expected to hold
for ample Cartier divisors. 
\end{remark}

\subsection{Upper bounds}

\begin{defn}Let $X$ be a projective variety and let $x\in X$. For any pseudo-effective $\bb R$-Cartier $\bb R$-divisor $L$ on $X$, define the Fujita--Nakayama-type invariant
$$\fuj(L;x):=\sup\{t\geq 0\st \pi^*L-tE\mbox{ is pseudo-effective}\},$$
where $\pi:\bl_xX\to X$ is the blow-up with exceptional divisor $E$.
\end{defn}

\begin{rmk}Let $x\in X$ be a smooth point. For any movable curve $C$ and pseudo-effective $\bb R$-Cartier $\bb R$-divisor $L$
we have $$\sh(C;x)\cdot\mu\,(L;x)\leq C\cdot L.$$
(This is because $(\pi^*C-\sh(C;x)\ell)\cdot(\pi^*L-\mu\,(L;x)E)\geq 0$.)
Note that $\mu\,(L;x)\geq \sh(L;x)$ when $L$ is nef.
\end{rmk}

Equality may be achieved in $\sh(C;x)\cdot\mu\,(L;x)\leq C\cdot L$

\begin{ex}[Picard rank 1]\label{ex:picard1}Let $X$ be a smooth projective variety of dimension $n$ with ${\rm rk}\, N^1(X)=1$.
Let $H$ be an ample generator of $N^1(X)$.
Let $x\in X$ and let $\widetilde X:=\bl_xX$ with blow-up morphism $\pi:\widetilde X\to X$. 
As usual denote the exceptional divisor by $E$ and the class of a line in $E=\bb P^{n-1}$ by $\ell$.
Then 
$$\Eff^1(\widetilde X)=\bigl\langle \pi^*H-\mu\,(H;x)E\ ,\ E\bigr\rangle\quad\mbox{and}\quad\Nef^1(\widetilde X)=\bigl\langle \pi^*H-\sh(H;x)E\ ,\ \pi^*H\bigr\rangle.$$ 
For curves we have 
$$\Eff_1(X)=\bigl\langle\pi^*H^{n-1}-\mu\, (H^{n-1};x)\ell\ ,\ \ell\bigr\rangle\quad\mbox{and}\quad \Mov_1(X)=\bigl\langle\pi^*H^{n-1}-\sh(H^{n-1};x)\ell\ ,\ \pi^*H^{n-1}\bigr\rangle.$$
We define $\mu\, (H^{n-1};x)$ as for divisors. The known dualities between these cones give
$$\mu\, (H^{n-1};x)\cdot\sh(H;x)=(H^n)\quad\mbox{and}\quad \sh(H^{n-1};x)\cdot\mu\, (H;x)=(H^n).$$\qed
\end{ex}

The inequality $\mu\,(L;x)\geq\sh(L;x)$ may be strict.

\begin{ex}When $X$ is an irreducible principally polarized abelian surface with Picard number 1 
and $H$ is a theta divisor,  
$$\sh(H;x)=\frac 43<\sqrt 2=\sqrt{(H^2)}<\frac{3}2=\mu\,(H;x)$$ for all $x\in X$.
See \cite[Proposition 2]{stef98}. \qed
\end{ex}

Example \ref{ex:picard1} extends partially to arbitrary rank.

\begin{ex}Let $X$ be a smooth projective variety. Let $[C]\in\Mov_1(X)$. 
For every $x\in X$ there exists $L_x$ a nonzero pseudo-effective divisor on $X$ such that
$\sh(C;x)\cdot\fuj(L_x;x)=C\cdot L_x$.
(If $\sh(C;x)=0$, use Theorem A to find $[L]\neq 0$ with $C\cdot L=0$. Assume henceforth that $\sh(C;x)>0$. Since $\pi^*[C]-\sh(C;x)\ell$ is not in the interior of the movable cone, there exists a 
pseudo-effective divisor $\overline L_x$ on $\bl_xX$ such that $(\pi^*[C]-\sh(C;x)\ell)\cdot \overline L_x=0$. Manifestly $[L_x]$ and $[E]$ are not proportional. 
In particular $L_x:=\pi_*\overline L_x$ is nonzero. 
We deduce that $\overline L_x=\pi^*L_x-\fuj(L_x;x)E$, and the conclusion follows.)
\end{ex}

\begin{ex}When the Picard rank is bigger than 1, it is not always the case that if $H$ is ample and $C=H^{n-1}$, then 
$\sh(C;x)\cdot\mu(H;x)=(H^n)$ for all $x\in X$. What could motivate the question is that $H$ minimizes the expression $\frac{C\cdot L}{\vol^{1/n}(L)}$ from the definition of $\mathfrak M\,(C)$ below.

Take $X=\bl_p\bb P^2$ and consider the ample divisor $3H-E$, where $H$ is the pullback of the line class from $\bb P^2$, and $E$ is the exceptional divisor.
We have $(3H-E)^2=8$.
If $x$ is one of the torus invariant points on $E$, then $\sh(3H-E;x)=1$ by Example \ref{ex:toric}, and $\mu(3H-E)=5$. The latter is because the effective cone of $\bl_xX$ is $\langle H-E-F,\ E-F,\ F\rangle$, where $F$ is the new exceptional line, and by $H$ and $E$ we denote the pullbacks of the respective classes from $X$.  \qed
\end{ex}

\begin{lem}Let $L$ be a big (for $\bb R$-Weil divisors, this is taken in the sense of \cite{fkl16}) $\bb R$-divisor on the projective variety $X$ of dimension $n$, and let $x\in X$ be a smooth point. Then $$\mu(L;x)\geq\vol^{1/n}(L).$$
\end{lem}

\begin{proof}There are ${{n+e-1}\choose{e-1}}\leq\frac{(e+n-1)^n}{n!}$ conditions for a member of a linear series on $X$
to vanish at $x$ with multiplicity at least $e$. Since $\vol(L)=\lim_{m\to\infty}\frac{\dim H^0(X;mL)}{m^n/n!}$, a limiting argument yields the result.
\end{proof}

\begin{cor}Let $X$ be a projective variety of dimension $n$, and let $x\in X$ be a smooth point.
Then $\sh(C;x)\leq\frac{C\cdot L}{\vol^{1/n}(L)}$ for all big $\bb R$-Cartier $\bb R$-divisors $L$
\end{cor}

\begin{rmk}\label{rmk:lx}Let $X$ be a smooth projective variety of dimension $n$, and let $\alpha\in\Mov_1(X)$. \cite{xi15}, \cite{lxpos} consider
$$\mathfrak M\,(\alpha):=\inf\left\{\left(\frac{\alpha\cdot L}{\vol^{1/n}(L)}\right)^{\frac n{n-1}}\st L\mbox{ is a big }\bb R\mbox{-Cartier }\bb R\mbox{-divisor}\right\}.$$
They prove that it can naturally be extended to a continuous function on $N_1(X)$
that is positive precisely in the interior of the dual cone $\Mov^1(X)^{\vee}$, which contains $\Mov_1(X)$. Note that $$\sh(\alpha;x)\leq\mathfrak M^{\frac{n-1}n}(\alpha)$$ for all $x\in X$.
\end{rmk}

\begin{cor}Let $X$ be a projective variety of dimension $n$, and let $x\in X$ be a smooth point.
Let $H$ be a big and nef $\bb R$-divisor class on $X$. Then $\sh([H^{n-1}];x)\leq (H^n)^{\frac{n-1}n}$. 
\end{cor}

\begin{ex}[Grassmann varieties]\label{ex:grass} Let $\ell$ be a line in the Grassmann variety $X=G(k,n)$ of $k$-dimensional subspaces of $\bb C^n$. Then $$\sh(\ell;x)=\frac 1{\min\{k,n-k\}}$$
for all $x\in X$. (Using $G(k,n)\simeq G(n-k,n)$, we may assume that $2k\leq n$.
Since $X$ is homogeneous, $\sh(\ell;x)$ is independent of $x$. Denote it $\sh(\ell)$.
Let $L$ be the very ample divisor induced by the Pl\" ucker embedding.
By \cite[Proposition 14.6.3]{fulton84} we have $\ell\cdot L=1$. 
Let $d:=(L^{k(n-k)})$ be the degree of the Grassmann variety.\footnote{
Though we don't use the formula here, it can be computed (for example by \cite[Proposition 1.10]{mukai}) as 
$d=\bigl(k(n-k)\bigr)!\prod_{1\leq i\leq k<j\leq n}(j-i)^{-1}.$}
We have $\ell=\frac 1dL^{k(n-k)-1}$. 
From Example \ref{ex:picard1} we deduce 
$\sh(\ell)=\frac{(L^{k(n-k)})}{d\cdot\mu\,(L)}=\frac 1{\mu\,(L)}$. 
The effective cone of the blow-up of the Grassmann variety at one point is computed in 
\cite[Corollary 3.2]{kop} and \cite[Lemma 7.2.2]{ri17}. It gives (also by \cite[Example 2]{RZ}) that
$\mu\,(L)=k.$ The conclusion follows.)\qed
\end{ex}

\begin{ex}[Curve in its Jacobian]
Let $C\subseteq J(C)$ be a smooth projective curve of genus $g\geq 1$ embedded
in its Jacobian with theta divisor $\theta$. Since $J(C)$ acts 
transitively on itself by translations, $\sh(C;x)$ is independent of $x$. We denote it $\sh(C)$.
By the previous corollary we find $$\sh(C)\leq\frac{g}{\sqrt[g]{g!}}< e,$$
where $e$ is Euler's constant.
(Recall that $[C]=\frac{\theta^{g-1}}{(g-1)!}$ and $(\theta^g)=g!$.)
By studying the singularities of $\theta$, we find $\sh(C)\leq \frac{g}{\nu\,(\theta)}\leq \frac{g}{[\sqrt g]},$
where $\nu\,(\theta)$ is the maximal multiplicity of $\theta$ at a point.
While the result is weaker than the asymptotic result, it is stronger than what the method above
could do for the 1-dimension linear series $|\theta|$. 
(By the Riemann singularity theorem this is the maximal dimension of $H^0(C;L)$ as $L$ ranges
over all effective divisors of degree $g-1$ in $C$. 
From the existence theorem in Brill--Noether theory (\cite[VII.(2.3) Theorem]{acgh} or \cite[Theorem 7.2.12]{laz04}), we know that $\nu\,(\theta)\geq[\sqrt g]$, while the Clifford index theorem implies $\nu\,(\theta)\leq \frac{g+1}2$. 
We have $\sh(C)\leq\frac{C\cdot\theta}{\nu\,(\theta)}=\frac g{\nu\,(\theta)}$.)

When $C$ is hyperelliptic of odd genus, then by the Clifford index theorem the bound
$\nu\,(\theta)=\frac{g+1}2$ is achievable by points corresponding to multiples of the unique $g^1_2$. We get the better bound
$$\sh(C)\leq\frac {2g}{g+1}<2<e.$$

When $C$ is not hyperelliptic of genus $g\geq 3$,
$$\sh(C)\leq\frac g2.$$
This result is only sharper than the first for $g=3$.
(A result of J. Fay (\cite[(1.2)]{pp01}) states that the elements of $|2\theta|$ with multiplicity at least 4 at the neutral element $o\in J(C)$
are precisely those that contain the difference variety $C-C$. The dimension of this subspace
is $2^g-1-{g\choose 2}$ (see \cite[(1.3)]{pp01}), so it is nonempty. Then $\sh(C)\leq\frac{C\cdot 2\theta}4=\frac{2g}4=\frac g2$.)
This result can be improved for some low genera by \cite[Chapter 6]{pp01}.
\end{ex}

By \cite{kong}, at least when $J(C)$ has Picard rank 1, the Seshadri constant $\sh(\theta)$
distinguishes between hyperelliptic and non-hyperelliptic curves. This is no longer the case for $\sh(C)$.

\begin{ex}[Curves of genus 3 whose Jacobian has Picard rank 1]\label{ex:g3r1} 
In this case we prove that $$\sh(C)=\frac 32.$$ 
(Assume first that $C$ is not hyperelliptic. The difference divisor $C-C$, which has multiplicity $2g-2=4$
at the origin and class $2\theta$, gives the upper bound $\sh(C)\leq\frac 32$. For the lower bound,
assume there exists an irreducible divisor $D$ of class $a\theta$ (necessarily proportional to $\theta$
since $J(C)$ has Picard rank 1 by assumption) with multiplicity $b$ at the origin $o\in J(C)$,
such that $\frac{C\cdot D}{\mult_oD}<\frac 32$, equivalently $b>2a$.
On the blow-up $\pi:\bl_oJ(C)\to J(C)$ with exceptional divisor $E$ we consider the product
$$\bigl(\pi^*\theta-\frac{12}7E\bigr)\cdot(\pi^*a\theta-bE)\cdot(\pi^*2\theta-4E)=12a-\frac{48}7b<12a-\frac{96}7a=-\frac{12}7a<0.$$
The first class is nef since $\sh(\theta)=\frac{12}7$ by \cite[Theorem 1.1.(2)]{kong}. The next two
classes are represented by the strict transforms of the distinct irreducible divisors $D$ and $C-C$, hence intersect properly.
Then the product of the 3 classes is nonnegative, which is a contradiction.
\vskip.25cm
When $C$ is hyperelliptic, then $C-C$ is a theta divisor (now the difference map $C\times C\to J(C)$ is generically 2-to-1 over its image). By the Clifford index theorem and the Riemann singularity theorem,
the singularities of any theta divisor have multiplicity 2. Then $C-C$ gives the bound $\sh(C)\leq\frac 32$. Consider the product on the blow-up of $J(C)$ at a singular point of $C-C$
$$\bigl(\pi^*\theta-\frac 32E)\cdot(\pi^*a\theta-bE)\cdot(\pi^*\theta-2E)=6a-3b<0,$$
where $a,b$ are as in the non-hyperelliptic case. The class $\pi^*\theta-\frac 32E$ is nef because
$\sh(\theta)=\frac 32$ by \cite[Theorem 1.1.(1)]{kong}. We get a contradiction as in the previous case.) \qed
\end{ex}

\begin{rmk}For very general (non-hyperelliptic) curves $C$ we have that ${\rm rk}N^1(J(C))=1$ (for example by \cite[Lemma on page 359]{acgh}).
This is also true for very general hyperelliptic curves of positive genus (The group $N^1(J(C))_{\bb Z}$ injects into ${\rm End}(J(C))$ by \cite[Corollary 19.1 and Corollary 19.2]{mumford}. Then \cite{zarhin}
and \cite{clark} prove that ${\rm End}(J(C))=\bb Z$ for very general hyperelliptic curves.)  
\end{rmk}

\begin{cor}If $C$ is a curve of genus 3 with ${\rm rk}\,N^1(J(C))=1$, let $\pi:X\to J(C)$ be the blow-up of the origin with exceptional divisor $E$. Then $\Eff^1(X)=\langle E,\pi^*\theta-2E\rangle$,
while $\Nef^1(X)=\langle\pi^*\theta,\pi^*\theta-\sigma E\rangle$, where $\sigma=\frac{12}7$ if $C$ is not hyperelliptic, and $\sigma=\frac 32$ if $C$ is hyperelliptic.
\end{cor}

\begin{proof}The boundary of the nef cone is determined by $\sh(\theta)$, which is computed in \cite{kong}. The boundary of the effective cone is determined by $\sh(C)$ as in Example \ref{ex:picard1}.
Use $[C]=\frac{\theta^2}{2}$ and $\theta^3=6$.
\end{proof}

\section{Seshadri constants for nef dual classes}

When $X$ is smooth projective of dimension $n$,
the space $N^{n-1}(X)$ is naturally identified with $N_1(X)$. The identification sends $\Nef^{n-1}(X)$
to $\Mov_1(X)$ by \cite[Theorem 11.4.19]{laz04}. In general we only have a linear surjection $N^{n-1}(X)\to N_1(X)$.





\begin{ex}Let $X\subseteq\bb P^N$ be a projective cone over a smooth projective variety
$Y\subseteq\bb P^{N-1}$ of dimension $n-1$. By \cite[Example 2.8]{flpos}, we have that 
$N_{k}(X)\simeq N_{k-1}(Y)$ for all $1\leq k\leq n-1$. In particular $N_1(X)\simeq\bb R$,
while $N_{n-1}(X)\simeq N^1(Y)$ can be arbitrarily large.
The cone $\Nef^{n-1}(X)$ is full-dimensional (\cite[Lemma 3.7]{flpos}) and salient/pointed/strict (\cite[Remark 2.14]{flpos}) in $N^{n-1}(X)$. In this case it surjects onto $\Mov_1(X)$, which is the non-negative
half line in $N_1(X)\simeq\bb R$.
\end{ex}

The definition of Seshadri constants for movable curves carries the inconvenience of 
asking that $L$ is Cartier. In particular, in the singular case, we can not restrict to irreducible
divisors as in Remark \ref{rmk:irredenough}. This difficulty is no longer present when considering 
dual classes, even in arbitrary codimension.

\begin{defn}Let $X$ be a projective variety, and let $\alpha\in\Nef^k(X)$.
For any $x\in X$ define
$$\sh(\alpha;x):=\inf\left\{\frac{\alpha\cdot Z}{\mult_xZ}\st Z\mbox{ effective }k\mbox{-cycle on }X\mbox{ containing x}\right\}.$$
We may restrict to the case where $Z$ is an irreducible subvariety of dimension $k$.
\end{defn}

Another advantage of dual classes over numerical classes is that they pullback.
Moreover, the projection formula shows that pulling back preserves nefness.

\begin{ex}When $k=1$ we recover the case of nef $\bb R$-Cartier $\bb R$-divisor.
When $k=n-1$ and $X$ is smooth we recover the case of movable curves.
\end{ex}

\begin{ex}Let $X$ be a possibly singular projective variety of dimension $n$ and choose $x\in X$.
Let $\alpha\in\Nef^{n-1}(X)$. Then $\alpha\cap[X]$ is a movable curve class.
We have $$\sh(\alpha;x)\leq\sh(\alpha\cap[X];x).$$ (The infimum on the right runs over Cartier divisors through $x$,
while one the one the left runs over the larger set of Weil divisors through $x$.)
\end{ex}

\begin{ex}[The inequality may be strict]\label{seshnefvsmov}Let $X\subseteq\bb P^N$ be a smooth projective surface, and let $C$ be a projective cone over $X$
with vertex $o$.
We have a (noncommutative) diagram
\begin{displaymath}
\xymatrixcolsep{4pc}\xymatrixrowsep{4pc}\xymatrix{Z\ar[r]^{\pi}\ar[d]_{\sigma} & X
\ar@{_{(}->}[dl]_{\imath}\\ C & }
\end{displaymath}
where $Z=\bb P_X(\cal O_X(1)\oplus\cal O_X)$ with projective bundle map $\pi:Z\to X$.
The morphism $\sigma$ is the blow-down of the section $X_0$ of $\pi$
corresponding to the projection on the second component $\cal O_X(1)\oplus\cal O_X\to\cal O_X$.
The embedding $\imath$ is the image via $\sigma$ of the section $X_1$ of $\pi$ corresponding to the projection
onto the first component $\cal O_X(1)\oplus\cal O_X\to\cal O_X(1)$. It is also the intersection of $C\subseteq\bb P^{N+1}$ with the $\bb P^N$ supporting $X$.
 
Let $\alpha\in\Nef^1(X)$. Since $\imath$ is the embedding of a very ample divisor,
it is not hard to see that the class $\imath_!\alpha\in N^2(C)$ defined by $\imath_!\alpha\cap[D]:=\alpha\cap\imath^*[D]$ for every $[D]\in N_2(C)$ is also nef.
We compute $$\sh(\imath_!\alpha;o)=\min\bigl\{t\geq 0\st \alpha-th\in\Nef^1(X)\bigr\},$$ 
where $h:=c_1(\cal O_X(1))$. 
(Let $D$ be an effective divisor on $C$ all of whose components pass through $o$.
Considering such divisors is sufficient for computing the Seshadri constant. 
Since $D$ and $\imath (X)$ meet properly, the class $\imath^*[D]$ is represented by a well defined effective 1-cycle $L$. 
Then $\imath_!\alpha\cap[D]=\alpha\cap\imath^*[D]=\alpha\cap[L]$.
We have $[D]=\sigma_*\pi^*[L]$. 
The strict transform of the cycle $D':=\sigma_*\pi^{-1}L$ is $\overline{D'}:=\pi^{-1}L$.
Then $\mult_oD'=-[\overline D']\cdot [X_0]^2=-\pi^*[L]\cdot [X_0]^2=h\cap [L]$. 
We prove that $\mult_oD'\geq\mult_oD$. Assuming this, and using that for every effective $L$ on $X$
the divisor $\sigma_*\pi^*L$ is effective and all of its components pass through $o$, we deduce 
$$\sh(\imath_!\alpha;o)=\inf\left\{\frac{\alpha\cap[L]}{h\cap[L]}\st L\mbox{ effective on }X\right\}
.$$
The claimed Seshadri constant computation is then a consequence of the duality between $\Nef^1(X)$ and $\Eff_1(X)$. For the claim on multiplicities, let $\overline D$ be the strict transform of $D$. Since $\sigma_*[\overline D]=[D]$, necessarily
$[\overline D]=\pi^*[L]+a[X_0]$ for some $a\in\bb R$. Since $\pi_*[\overline D]=a[X]$ is effective on $X$, necessarily $a\geq 0$. Then $\mult_oD=-[\overline D]\cdot[X_0]^2=(\pi^*[L]+a[X_0])\cdot\pi^*h\cdot[X_0]=([L]-ah)\cdot h\leq [L]\cdot h$.)
\vskip.25cm
Next we compute
$$\sh(\imath_!\alpha\cap[C];o)=\frac{\alpha\cdot h}{h^2}.$$
(Every nonzero effective Cartier divisor $D$ on $C$ is equivalent to $a\imath(X)$ for some $a>0$.
Furthermore $\imath^*[D]=ah$.
Then $(\imath_!\alpha\cap[C])\cap[D]=\alpha\cap(\imath^*[D]\cap[X])=\alpha\cdot h$.
For an arbitrary effective divisor $L$ with $\cal O_X(L)\simeq\cal O_X(a)$, the divisor $D':=\sigma_*\pi^*L$
is Cartier on $C$ passing through $o$. The multiplicity at $o$ is $-\pi^*L\cdot[X_0]^2=ah^2$.
As before, we prove $\mult_oD'\geq\mult_oD$. 
Let $\overline D$ be the strict transform of $D$ on $Z$. 
The only classes that push to $[D]$ have form $\pi^*\imath^*[D]+b[X_0]$ for some $b$.
Set then $[\overline D]=a\pi^*h+b[X_0]$. By pushing to $X$, we find $b\geq 0$.
Using that $\overline D$ and $X_0$ meet properly, we deduce that $(a-b)h$ is effective, so $a\geq b$.
It follows that $\mult_oD=-[\overline D]\cdot[X_0]^2=(a\pi^*h+b[X_0])\cdot\pi^*h\cdot[X_0]=(a-b)h^2\leq ah^2=\mult_oD'$. The formula for the Seshadri constant follows easily.)
\vskip.25cm
If $X$ is the blow-up of $\bb P^2$ at one point and $\alpha$ is the pullback of a line class from $\bb P^2$, then $\alpha\cdot L=0$ for $L$ the exceptional $\bb P^1$ and we find
$\sh(\imath_!\alpha;o)=0<\sh(\imath_!\alpha\cap[C];o)$ for any ample $h$ on $X$.
\qed
\end{ex}

\begin{remark}In the previous example, we see that $\sh(\imath_!\alpha;o)$ is a more subtle invariant
than $\sh(\imath_!\alpha\cap[C];o)$. This suggests that on singular varieties it may be more
fruitful to study nef ``dual'' classes than movable classes.   
\end{remark}

\begin{ex}By comparison, Seshadri constants at the vertex of a cone are less subtle for nef divisors.
In fact $\sh([X_1];o)=1$. (Indeed if $\ell$ is a line through the vertex, then $\frac{[X_1]\cdot\ell}{\mult_o\ell}=\frac 11=1$.
For any other curve $T$ through $o$, we have $\frac{[X_1]\cdot T}{\mult_o T}=\frac{[X_1]\cdot\overline T}{[X_0]\cdot\overline T}$, 
where $\overline T$ denotes the strict transform of $T$ on $Z$. However $([X_1]-[X_0])\cdot\overline T=h\cdot\pi_*\overline T>0$,
leading to $\frac{[X_1]\cdot T}{\mult_o T}>1$.)\qed
\end{ex}

\begin{ex}[Irrational Seshadri constant for nef classes]Let $X$ be an abelian surface with a round
nef cone as in \cite[Example 2.3.8]{laz04}. Then for general choices of integral classes $h$ and $\alpha$, the number
$\min\{t\st \alpha-th\in\Nef^1(X)\}=\sh(\imath_!\alpha;o)$ is quadratic irrational. 
\vskip.2cm
\noindent {\bf N.B.} It is not obvious what one should mean by an ``integral'' dual class in $N^k(X)$ (coming from a weight $k$ Chern polynomial with integer coefficients vs. taking integer values on $N_k(X)$). The class $\imath_!\alpha$ takes integer values on $N_2(C)_{\bb Z}$, 
so with either definition it is at least rational.\qed
\end{ex}

\begin{rmk}For fixed $x\in X$, the function $\alpha\mapsto\sh(\alpha;x)$ is $1$-homogeneous, nonnegative, and concave on $\Nef^k(X)$.
\end{rmk}

\begin{rmk}
Semi-continuity-type statements analogous to Proposition \ref{prop:seshadrisemicontinuity} and Lemma \ref{lem:movingjumps} 
hold for nef classes as well.
\end{rmk}

\begin{ex}[Toric varieties] Let $X=X(\Delta)$ be a projective toric variety, possibly singular and let $x_{\sigma}$ be a torus-invariant point. 
Let $\alpha\in\Nef^k(X)$. Then
$$\sh(\alpha;x_{\sigma})=\min\left\{\frac{\alpha\cdot V_{\theta}}{\mult_{x_{\sigma}}V_{\theta}}\st \theta\mbox{ is an }(n-k)\mbox{-dimensional subcone of }\sigma\right\}.$$
Here $V_{\theta}$ denotes the $k$-dimensional torus-invariant subvariety of $X$
corresponding to $\theta$. (Multiplicity is upper semi-continuous in families, so the maximal multiplicity at $x_{\sigma}$ of effective cycles in any given class
is achieved by torus-invariant ones.) 
\end{ex}

\begin{lem}\label{lem:nefkinteriorlowerbound}If $\alpha$ is in the strict interior of $\Nef^k(X)$, then there exists
a constant $\varepsilon>0$ such that $\sh(\alpha;x)\geq\varepsilon$ for all $x\in X$.
\end{lem}

\begin{proof}By \cite[Corollary 3.15]{flpos}, complete intersections are in the strict interior of $\Nef^k(X)$. Thus by concavity and nonnegativity it is enough to consider the case where 
$\alpha=[H^{k}]$ for some ample divisor $H$ on $X$. This case is treated analogously
to Lemma \ref{lem:uniformlowerbound}.
\end{proof}

When $x$ is a smooth point, we could also deduce the previous lemma from the following

\begin{lem}\label{lem:seshadriproduct}Let $x\in X$ be a smooth point. Let $\alpha\in\Nef^k(X)$, and let $H$  be an ample $\bb R$-divisor class. Then $\sh(\alpha\cdot H;x)\geq\sh(\alpha;x)\sh(H;x)$.
\end{lem}

\begin{proof}Jet separation for divisors implies that for every $k+1$-dimensional subvariety $Z$ through $x$ we can find an $\bb R$-divisor of class $H$ through $x$ meeting $Z$ properly and with multiplicity arbitrarily close to $\sh(H;x)$ at $x$.
The intersection $H\cdot Z$ is then represented by the limit of (the classes of the elements of) a sequence of effective $k$-cycles $Y_m$ with $\mult_xY_m\geq \bigl(\sh(H;x)-\frac 1m\bigr)\mult_xZ$. The last part uses that $x\in X$ is smooth. From $\frac{(\alpha\cdot H)\cdot Z}{\mult_xZ}\geq \frac{\alpha\cdot (H\cdot Z)}{\mult_xY_m}\cdot\bigl(\sh(H;x)-\frac 1m\bigr)\geq\sh(\alpha;x)\bigl(\sh(H;x)-\frac 1m\bigr)$ we conclude by taking 
$m\to\infty$ and then the infimum over all $Z$.
\end{proof}

As we did for movable curves we deduce

\begin{cor}The function $\sh(\cdot;x)$ is locally uniformly continuous on the strict interior of $\Nef^k(X)$.
\end{cor}

\begin{conj}[Seshadri criterion] Let $X$ be a projective variety. Let $\alpha\in\Nef^k(X)$
such that there exists $\varepsilon>0$ with $\sh(\alpha;x)\geq\varepsilon$ for all $x\in X$.
Then $\alpha$ is in the strict interior of $\Nef^k(X)$.
\end{conj}

We verify this for curves.

\begin{prop}\label{prop:seshadrinefdualcurves}Let $X$ be a projective variety of dimension $n$ over an algebraically closed field.
Let $\alpha\in\Nef^{n-1}(X)$ such that there exists $\varepsilon>0$ such that
$\sh(\alpha;x)\geq\varepsilon$ for all $x\in X$. Then $\alpha$ is in the strict interior of $\Nef^{n-1}(X)$.
\end{prop}

Note that the converse is provided by Lemma \ref{lem:nefkinteriorlowerbound}.
The case when $X$ is smooth is covered by Theorem A, since in this case 
$N^{n-1}(X)=N_1(X)$ and $\Nef^{n-1}(X)=\Mov_1(X)$. The proof is similar to Theorem A.

\begin{proof}If $\alpha$ is not an interior class, then there exists $0\neq[L]\in\Eff_{n-1}(X)$
such that $\alpha\cdot L=0$. Let $\pi:\widetilde X\to X$ be a nonsingular alteration.
By \cite[Corollary 3.22]{flpos} there exists $[\bar L]\in\Eff_{n-1}(\widetilde X)$ with $\pi_*[\bar L]=[L]$.
Consider the divisorial Zariski decomposition $\bar L=P_{\sigma}(\bar L)+N_{\sigma}(\bar L)$.
From the projection formula and the nefness of $\alpha$, we deduce $\pi^*\alpha\cdot P_{\sigma}(\bar L)=0$ and $\alpha\cdot\pi_*N_{\sigma}(\bar L)=0$.
Note that $N_{\sigma}(\bar L)$ is an effective $\bb R$-divisor. 
If $\pi_*N_{\sigma}(\bar L)\neq 0$, for any point $x$ in its support we find a contradiction
$0=\alpha\cdot\pi_*N_{\sigma}(\bar L)\geq\varepsilon\mult_x\pi_*N_{\sigma}(\bar L)>0$. 
We deduce that $\pi_*N_{\sigma}(\bar L)=0$. Since $L$ is not numerically trivial, neither is $P_{\sigma}(\bar L)$. It is movable, and hence there exists $A$ ample on $\widetilde X$ such that
$\dim H^0\bigl(\widetilde X,\cal O_{\widetilde X}(\lfloor mP_{\sigma}(\bar L)\rfloor+A)\bigr)$ grows at least linearly with $m$. Let $x\in X$ be a point in the finite locus of $\pi$, and let $\tilde x$ be a closed point in $\pi^{-1}\{x\}$. Let $D_m\in|\lfloor mP_{\sigma}(\bar L)\rfloor+A|$ with $\lim_{m\to\infty}\mult_{\tilde x}D_m=\infty$.  Then $\lim_{m\to\infty}\mult_x\pi_*D_m=\infty$.
In fact $\mult_x\pi_*D_m\geq\mult_{\tilde x}D_m$ by Lemme \ref{lem:multiplicitypush}. As in Theorem A we find that 
$\alpha\cdot(mL+\pi_*A)=\pi^*\alpha\cdot(mP_{\sigma}(\bar L)+A)\geq\pi^*\alpha\cdot D_m=\alpha\cdot \pi_*D_m\geq\varepsilon\cdot\mult_x\pi_*D_m$ grows to infinity.
This is impossible, since $\alpha\cdot L=0$. 
\end{proof}

\begin{prop}Let $X$ be a projective variety, and let $x\in X$ be a possibly singular point.
Let $\alpha\in\Nef^k(X)$. Let $E$ be the exceptional divisor of the blow-up $\pi:\bl_xX\to X$.
Then 
$$\sh(\alpha;x)=\max\bigl\{t\st\pi^*\alpha+t(-E)^k\in\Nef^k(\bl_xX)\bigr\}.$$
\end{prop}  

\begin{proof}The proof is analogous to Proposition \ref{prop:seshadriviapull}.
The only part that may require additional explanation is that $(\pi^*\alpha+t(-E)^k)\cdot Z\geq 0$
when $Z$ is an effective cycle mapped to $x$ by $\pi$ and $t\geq 0$. This is because $-E|_Z$ is ample,
and because $\pi^*\alpha\cdot Z=0$ by the projection formula.
\end{proof}

\begin{lem}Let $\pi:X\to Y$ be a generically finite dominant morphism of projective varieties.
Let $\alpha\in\Nef^k(Y)$. Then $\sh(\pi^*\alpha;x)\geq\sh(\alpha;\pi(x))$ for all $x$ such
that $\pi$ is finite around $x$.
\end{lem}

\begin{proof}Let $Z$ be a $k$-dimensional subvariety through $x$. By Lemma \ref{lem:multiplicitypush} we have $\mult_{\pi(x)}\pi_*Z\geq\mult_xZ$,
hence $\frac{\pi^*\alpha\cdot[Z]}{\mult_xZ}\geq\frac{\alpha\cdot[\pi_*Z]}{\mult_{\pi(x)}\pi_*Z}\geq\sh(\alpha;\pi(x))$.
\end{proof}

\begin{prop}[Null locus]\label{prop:singularnulllocus} Let $X$ be a projective variety of dimension $n$,
and let $\alpha\in\Nef^{n-1}(X)$ contained in the strict interior of $\Mov_{n-1}(X)^{\vee}$.
Then
\begin{enumerate}[i)]
\item If $[L]\in\Eff_{n-1}(X)$ satisfies $\alpha\cdot[L]=0$, then $L$ is numerically equivalent to
an effective divisor. Furthermore there are only finitely many irreducible effective divisors $L_1,\ldots,L_r$
such that $\alpha\cdot[L_i]=0$.
\item Set ${\rm Null}\,(\alpha):=L_1\cup\ldots\cup L_r$, and $SV\,(\alpha):=\{x\in X\st \sh(\alpha;x)=0\}$. Then
$${\rm Null}\,(\alpha)\subseteq SV\,(\alpha)\subseteq {\rm Null}\,(\alpha)\cup \Sing X.$$
\end{enumerate}
\end{prop}

\begin{proof} $i)$. Let $\pi:\widetilde X\to X$ be a nonsingular alteration. We claim that $\pi^*\alpha$
is also in the strict interior of $\Mov_{n-1}(\widetilde X)^{\vee}$. If not, since $\pi^*\alpha$ is nef, there exists $[\widetilde P]\neq 0$ movable on $\widetilde X$ with $\pi^*\alpha\cdot[\widetilde P]=0$. Since $\pi_*\widetilde P\in\Mov_{n-1}(X)$, and $\alpha$ is in the strict interior of the dual cone,
by the projection formula we find that necessarily $\pi_*[\widetilde P]=0$.
Let $H$ be an ample divisor on $X$. By the projection formula $[\widetilde P]\cdot \pi^*H^{n-1}=0$.
A repeated application of \cite[Corollary 3.11, Lemma 3.12]{fl13z} proves $[\widetilde P]=0$.

Assume now $\alpha\cdot[L]=0$. By \cite[Corollary 3.22]{flpos} there exists $[\bar L]\in\Eff_{n-1}(\widetilde X)$ with $\pi_*[\bar L]=[L]$. By the projection formula $\pi^*\alpha\cdot[\bar L]=0$.
Since $\pi^*\alpha$ is in the interior of $\Mov_{n-1}(\widetilde X)^{\vee}$, we deduce that
$\bar L\equiv N_{\sigma}(L)$. Pushing to $X$, it follows that $L\equiv \pi_*N_{\sigma}(\bar L)$.
The latter is an effective divisor as desired. The finiteness statement follows from 
Theorem B.

$ii)$. Clearly $\sh(\alpha;x)=0$ for $x\in L_i$. Assume now that a smooth point $x\in X$ satisfies $\sh(\alpha;x)=0$. 
Let $D_i$ be a sequence of irreducible divisors through $x$ such that 
$\lim_i\frac{\alpha\cdot[D_i]}{\mult_xD_i}=0$. Let $H$ be a very ample divisor on $X$.
As in Theorem B we may assume $|D_i|=1$ (that is $H^{n-1}\cdot[D_i]=1$), and that $\lim_i[D_i]=[D]$ for some (nonzero)
$[D]\in\Eff_{n-1}(X)$ with $\alpha\cdot[D]=0$. 
Let $\pi:\widetilde X\to X$ be a nonsingular alteration. 
Let $\widetilde D_i$ be a divisor with irreducible support on $\widetilde X$
with $\pi_*\widetilde D_i=D_i$. We claim that $[\widetilde D_i]$ form a bounded sequence in $N^1(\widetilde X)$. Assuming this, by passing to a subsequence, we may assume
that $\lim_i[\widetilde D_i]=[\widetilde D]$ for some $\widetilde D$ with $\pi_*[\widetilde D]=[D]$.
By the semicontinuity arguments in the proof of Theorem B, we may assume that $\widetilde D_i$
and so also $D_i$ is a constant sequence. It follows that $[D]$ is represented by a
divisor with irreducible support containing $x$. Since $\alpha\cdot[D]=0$, the support is one of the $L_i$.

For the claim, since $\pi^*H$ is big, there exist divisors $A$ and $E$ on $\widetilde X$
with $A$ ample and $E$ effective such that $\pi^*H=A+E$. If $\widetilde D_i$ meets $E$ properly, then by \cite[Lemma 4.12]{fl16sw}, we have
$A^{n-1}\cdot [\widetilde D_i]\leq \pi^*H^{n-1}\cdot[\widetilde D_i]=H^{n-1}\cdot D_i=1$.
Since multiplying by $A^{n-1}$ is a norm (\cite[Theorem 1.4.(3)]{flpos}), the conclusion follows.
If $\widetilde D_i$ does not meet $E$ properly, then its (irreducible) support is contained in $\Supp E$,
so it is one of the irreducible components of $\Supp E$. From $|D_i|=1$, we again deduce
that up to passing to a subsequence $\widetilde D_i$ is constant, hence so is $D_i$.
\end{proof}

Unlike in Theorem \ref{thm:nulllocus}.iii), if $k$ is arbitrary and $\alpha\in\Nef^k(X)$, it is not generally true that the existence of $x_0\in X$ with $\sh(\alpha;x_0)>0$ implies that $\alpha$ is big. 

\begin{ex}\cite{delv11,OttemNef} construct examples of projective $4$-folds $X$ where the unexpected inclusion $\Eff_2(X)\subsetneq\Nef^2(X)$ holds. In particular there exist pseudo-effective non-big classes $\alpha$ in the strict interior of $\Nef^2(X)$. In fact, for such classes we have $\inf_{x\in X}\sh(\alpha;x)>0$ by Lemma \ref{lem:nefkinteriorlowerbound}.

Let $X\subset G(2,6)$ be the variety of lines on a very general cubic fourfold in $\bb P^5$.
With notation as in \cite{OttemNef}, let $\cal U^{\vee}$ be the tautological rank 2 quotient bundle on $G(2,6)$, and denote $g:=c_1(\cal U^{\vee})$ and $c:=c_2(\cal U^{\vee})$. 
The class $c$ is extremal in $\Eff_2(X)$ by \cite[Proposition 2.4]{voisin}.
The proof of \cite[Theorem 1]{OttemNef} shows that $20c-g^2$ is nef. 
The concavity of Seshadri constants and Lemma \ref{lem:seshadriproduct} imply that $\sh(c;x)\geq\frac 1{20}\sh(g^2;x)\geq\frac 1{20}\sh(g;x)^2\geq\frac 1{20}$ for all $x\in X$.\qed
\end{ex}

\begin{remark}The same examples of nef non-pseudo-effective classes show that we cannot expect a 
geometric interpretation of Seshadri constants for all classes in the interior of $\Nef^k(X)$ as in Theorem C.  
\end{remark}

\nocite{*}
\bibliographystyle{amsalpha}
\bibliography{Seshadri}

\newcommand{\etalchar}[1]{$^{#1}$}
\providecommand{\bysame}{\leavevmode\hbox to3em{\hrulefill}\thinspace}
\providecommand{\MR}{\relax\ifhmode\unskip\space\fi MR }
\providecommand{\MRhref}[2]{%
  \href{http://www.ams.org/mathscinet-getitem?mr=#1}{#2}
}
\providecommand{\href}[2]{#2}
\begin{thebibliography}{BDRH{\etalchar{+}}09}

\bibitem[ACGH85]{acgh}
E.~Arbarello, M.~Cornalba, P.~A. Griffiths, and J.~Harris, \emph{Geometry of
  algebraic curves. {V}ol. {I}}, Grundlehren der Mathematischen Wissenschaften
  [Fundamental Principles of Mathematical Sciences], vol. 267, Springer-Verlag,
  New York, 1985. \MR{770932}

\bibitem[BDPP13]{bdpp13}
S{\'e}bastien Boucksom, Jean-Pierre Demailly, Mihai P{\u{a}}un, and Thomas
  Peternell, \emph{The pseudo-effective cone of a compact {K}\"ahler manifold
  and varieties of negative {K}odaira dimension}, J. Algebraic Geom.
  \textbf{22} (2013), no.~2, 201--248.

\bibitem[BDRH{\etalchar{+}}09]{primer}
Thomas Bauer, Sandra Di~Rocco, Brian Harbourne, Micha\l{} Kapustka, Andreas
  Knutsen, Wioletta Syzdek, and Tomasz Szemberg, \emph{A primer on {S}eshadri
  constants}, Interactions of classical and numerical algebraic geometry,
  Contemp. Math., vol. 496, Amer. Math. Soc., Providence, RI, 2009, pp.~33--70.
  \MR{2555949}

\bibitem[BFJ09]{bfj}
S\'ebastien Boucksom, Charles Favre, and Mattias Jonsson,
  \emph{Differentiability of volumes of divisors and a problem of {T}eissier},
  J. Algebraic Geom. \textbf{18} (2009), no.~2, 279--308. \MR{2475816}

\bibitem[CHMS14]{chms14}
Paolo Cascini, Christopher Hacon, Mircea Musta\c{t}\u{a}, and Karl Schwede,
  \emph{On the numerical dimension of pseudo-effective divisors in positive
  characteristic}, Amer. J. Math. \textbf{136} (2014), no.~6, 1609--1628.
  \MR{3282982}

\bibitem[Cla11]{clark}
Pete~L. Clark, \emph{Endomorphisms of the jacobian of a general hyperelliptic
  curve}, Mathematics Stack Exchange, 2011,
  URL:https://math.stackexchange.com/q/49915 (version: 2011-07-06).

\bibitem[Cut15]{Cutkosky}
Steven~Dale Cutkosky, \emph{Teissier's problem on inequalities of nef
  divisors}, J. Algebra Appl. \textbf{14} (2015), no.~9, 1540002, 37.
  \MR{3368254}

\bibitem[DELV11]{delv11}
Olivier Debarre, Lawrence Ein, Robert Lazarsfeld, and Claire Voisin,
  \emph{Pseudoeffective and nef classes on abelian varieties}, Compos. Math.
  \textbf{147} (2011), no.~6, 1793--1818.

\bibitem[Dem92]{dem}
Jean-Pierre Demailly, \emph{Singular {H}ermitian metrics on positive line
  bundles}, Complex algebraic varieties ({B}ayreuth, 1990), Lecture Notes in
  Math., vol. 1507, Springer, Berlin, 1992, pp.~87--104. \MR{1178721}

\bibitem[dJ96]{dejong96}
Aise~Johan de~Jong, \emph{Smoothness, semi-stability and alterations}, Inst.
  Hautes \'Etudes Sci. Publ. Math. (1996), no.~83, 51--93.

\bibitem[EKL95]{ekl95}
Lawrence Ein, Oliver K{\"u}chle, and Robert Lazarsfeld, \emph{Local positivity
  of ample line bundles}, J. Differential Geom. \textbf{42} (1995), no.~2,
  193--219.

\bibitem[EL93]{el93}
Lawrence Ein and Robert Lazarsfeld, \emph{Seshadri constants on smooth
  surfaces}, Ast\'erisque (1993), no.~218, 177--186, Journ{\'e}es de
  G{\'e}om{\'e}trie Alg{\'e}brique d'Orsay (Orsay, 1992).

\bibitem[ELM{\etalchar{+}}09]{ELMNP}
Lawrence Ein, Robert Lazarsfeld, Mircea Musta\c{t}\u{a}, Michael Nakamaye, and
  Mihnea Popa, \emph{Restricted volumes and base loci of linear series}, Amer.
  J. Math. \textbf{131} (2009), no.~3, 607--651. \MR{2530849}

\bibitem[FKL16]{fkl16}
Mihai Fulger, J\'anos Koll\'ar, and Brian Lehmann, \emph{Volume and {H}ilbert
  functions of {$\Bbb R$}-divisors}, Michigan Math. J. \textbf{65} (2016),
  no.~2, 371--387. \MR{3510912}

\bibitem[FL83]{fl83}
William Fulton and Robert Lazarsfeld, \emph{Positive polynomials for ample
  vector bundles}, Ann. of Math. (2) \textbf{118} (1983), no.~1, 35--60.

\bibitem[FL16]{fl16sw}
Mihai Fulger and Brian Lehmann, \emph{Morphisms and faces of pseudo-effective
  cones}, Proc. Lond. Math. Soc. (3) \textbf{112} (2016), no.~4, 651--676.
  \MR{3483128}

\bibitem[FL17a]{flpos}
\bysame, \emph{Positive cones of dual cycle classes}, Algebraic Geometry
  \textbf{4} (2017), no.~1, 1--28.

\bibitem[FL17b]{fl13z}
\bysame, \emph{Zariski decompositions of numerical cycle classes}, J. Algebraic
  Geom. \textbf{26} (2017), no.~1, 43--106. \MR{3570583}

\bibitem[Fuj09]{fujino09}
Osamu Fujino, \emph{Big {R}-divisors}, 2009,
  https://www.math.kyoto-u.ac.jp/$\sim$fujino/big-r-divisor5.pdf.

\bibitem[Ful84]{fulton84}
William Fulton, \emph{Intersection theory}, Ergebnisse der Mathematik und ihrer
  Grenzgebiete (3) [Results in Mathematics and Related Areas (3)], vol.~2,
  Springer-Verlag, Berlin, 1984.

\bibitem[Ful11]{ful11}
Mihai Fulger, \emph{Cones of effective cycles on projective bundles over
  curves}, Math. Z. \textbf{269} (2011), no.~1-2, 449--459.

\bibitem[Har70]{Har70}
Robin Hartshorne, \emph{Ample subvarieties of algebraic varieties}, Lecture
  Notes in Mathematics, Vol. 156, Springer-Verlag, Berlin-New York, 1970, Notes
  written in collaboration with C. Musili.

\bibitem[Kle66]{kleiman66}
Steven~L. Kleiman, \emph{Toward a numerical theory of ampleness}, Ann. of Math.
  (2) \textbf{84} (1966), 293--344.

\bibitem[Kon03]{kong}
Jian Kong, \emph{Seshadri constants on {J}acobian of curves}, Trans. Amer.
  Math. Soc. \textbf{355} (2003), no.~8, 3175--3180. \MR{1974680}

\bibitem[Kop16]{kop}
John Kopper, \emph{Effective cycles on blow-ups of grassmannians}, 2016,
  arXiv:1612.01906 [math.AG].

\bibitem[Lau17]{Lau17}
Chung-Ching Lau, \emph{Numerical dimension and locally ample curves}, 2017,
  arXiv:1703.08265 [math.AG].

\bibitem[Laz04]{laz04}
Robert Lazarsfeld, \emph{Positivity in algebraic geometry}, Ergebnisse der
  Mathematik und ihrer Grenzgebiete. 3. Folge. A Series of Modern Surveys in
  Mathematics [Results in Mathematics and Related Areas. 3rd Series. A Series
  of Modern Surveys in Mathematics], vol.~48, Springer-Verlag, Berlin, 2004,
  Classical setting: line bundles and linear series.

\bibitem[Leh16]{Leh16}
Brian Lehmann, \emph{Volume-type functions for numerical cycle classes}, Duke
  Math. J. \textbf{165} (2016), no.~16, 3147--3187. \MR{3566200}

\bibitem[LM09]{LM09}
Robert Lazarsfeld and Mircea Musta\c{t}\u{a}, \emph{Convex bodies associated to
  linear series}, Ann. Sci. \'Ec. Norm. Sup\'er. (4) \textbf{42} (2009), no.~5,
  783--835. \MR{2571958}

\bibitem[LX16]{lxpos}
Brian Lehmann and Jian Xiao, \emph{Positivity functions for curves on algebraic
  varieties}, 2016, arXiv:1607.05337 [math.AG].

\bibitem[Muk93]{mukai}
Shigeru Mukai, \emph{Curves and {G}rassmannians}, Algebraic geometry and
  related topics ({I}nchon, 1992), Conf. Proc. Lecture Notes Algebraic Geom.,
  I, Int. Press, Cambridge, MA, 1993, pp.~19--40. \MR{1285374}

\bibitem[Mum08]{mumford}
David Mumford, \emph{Abelian varieties}, Tata Institute of Fundamental Research
  Studies in Mathematics, vol.~5, Published for the Tata Institute of
  Fundamental Research, Bombay; by Hindustan Book Agency, New Delhi, 2008, With
  appendices by C. P. Ramanujam and Yuri Manin, Corrected reprint of the second
  (1974) edition. \MR{2514037}

\bibitem[Mus13]{mustata}
Mircea Musta\c{t}\u{a}, \emph{The non-nef locus in positive characteristic}, A
  celebration of algebraic geometry, Clay Math. Proc., vol.~18, Amer. Math.
  Soc., Providence, RI, 2013, pp.~535--551. \MR{3114955}

\bibitem[MX17]{mx17}
Nicholas McCleerey and Jian Xiao, \emph{Polar transform and local positivity
  for curves}, 2017, arXiv:1707.06349 [math.AG].

\bibitem[Nak04]{nak}
Noboru Nakayama, \emph{Zariski-decomposition and abundance}, MSJ Memoirs,
  vol.~14, Mathematical Society of Japan, Tokyo, 2004. \MR{2104208}

\bibitem[Ott15]{OttemNef}
John~Christian Ottem, \emph{Nef cycles on some hyperkahler fourfolds}, 2015,
  arXiv:1505.01477 [math.AG].

\bibitem[Ott16]{Ottem16}
\bysame, \emph{On subvarieties with ample normal bundle}, J. Eur. Math. Soc.
  (JEMS) \textbf{18} (2016), no.~11, 2459--2468. \MR{3562347}

\bibitem[PP01]{pp01}
Christian Pauly and Emma Previato, \emph{Singularities of {$2\Theta$}-divisors
  in the {J}acobian}, Bull. Soc. Math. France \textbf{129} (2001), no.~3,
  449--485. \MR{1881203}

\bibitem[Ris17]{ri17}
Rick Rischter, \emph{Projective and birational geometry of grassmannians and
  other special varieties}, 2017, arXiv:1705.05673 [math.AG].

\bibitem[RZ01]{RZ}
Joachim Rosenthal and Andrei Zelevinsky, \emph{Multiplicities of points on
  {S}chubert varieties in {G}rassmannians}, J. Algebraic Combin. \textbf{13}
  (2001), no.~2, 213--218. \MR{1826954}

\bibitem[Ste98]{stef98}
A.~Steffens, \emph{Remarks on {S}eshadri constants}, Math. Z. \textbf{227}
  (1998), no.~3, 505--510.

\bibitem[Tak07]{Takagi}
Satoshi Takagi, \emph{Fujita's approximation theorem in positive
  characteristics}, J. Math. Kyoto Univ. \textbf{47} (2007), no.~1, 179--202.
  \MR{2359108}

\bibitem[Voi10]{voisin}
Claire Voisin, \emph{Coniveau 2 complete intersections and effective cones},
  Geom. Funct. Anal. \textbf{19} (2010), no.~5, 1494--1513. \MR{2585582}

\bibitem[Xia15]{xi15}
Jian Xiao, \emph{Characterizing volume via cone duality}, 2015,
  arXiv:1502.06450 [math.AG].

\bibitem[Zar00]{zarhin}
Yuri~G. Zarhin, \emph{Hyperelliptic {J}acobians without complex
  multiplication}, Math. Res. Lett. \textbf{7} (2000), no.~1, 123--132.
  \MR{1748293}

\end{thebibliography}

\end{document}